\documentclass[reqno]{amsart}
\usepackage{amsmath,amssymb,amsfonts,caption}
\usepackage{graphicx}
\usepackage{euscript}
\usepackage{enumerate}
\usepackage{verbatim}
\usepackage{epsfig}

\usepackage{color}
\usepackage[usenames,dvipsnames,svgnames,table]{xcolor}
\definecolor{orange}{rgb}{1,0.5,0}
\newtheorem{theorem}{Theorem}
\newtheorem{corollary}{Corollary}

\newtheorem{lemma}{Lemma}

\renewcommand{\epsilon}{\varepsilon}
\renewcommand{\phi}{\varphi}
\renewcommand{\le}{\leqslant}
\renewcommand{\ge}{\geqslant}
\newcommand{\eps}{\varepsilon}
\newcommand{\lbd}{\lambda}

\newcommand{\cR}{\mathcal R}

\usepackage{dsfont}
   \newcommand{\N}{\ensuremath{\mathds N}}
   \newcommand{\R}{\ensuremath{\mathds R}}

\begin{document}
\title[]
   {Dynamics of a discrete eco-epidemiological model with disease in the prey}
\author{Lopo F. de Jesus}
\address{L. F. de Jesus
   Departamento de Matem\'atica\\
   Universidade da Beira Interior\\
   6201-001 Covilh\~a\\
   Portugal}
   \email{lopo.jesus@ubi.pt}
\author{C\'esar M. Silva}
\address{C. M. Silva\\
   Departamento de Matem\'atica\\
   Universidade da Beira Interior\\
   6201-001 Covilh\~a\\
   Portugal}
\email{csilva@ubi.pt}
\author{Helder Vilarinho}
\address{H. Vilarinho\\
   Departamento de Matem\'atica\\
   Universidade da Beira Interior\\
   6201-001 Covilh\~a\\
   Portugal}
\email{helder@ubi.pt}
\date{\today}
\thanks{L. de Jesus, C. M. Silva and H. Vilarinho were partially supported by FCT through CMA-UBI (project UIDB/MAT/00212/2020).}
\begin{abstract}
 Using Mickens nonstandard method, we obtain a discrete family of nonautonomous eco-epidemiological models that include general functions corresponding to the predation of the infected and uninfected preys. We obtain results on the persistence and extinction of the infected preys assuming that the bi-dimensional predator-prey subsystem that describes the dynamics in the absence of the infection satisfies some assumptions. Some examples and simulations are undertaken to {illustrate} our results.
\end{abstract}
\maketitle
\section{Introduction}
{In many situations eco-epidemiological models describe more accurately some ecological system than classical Lotka-Volterra models where the disease is not taken into account. It is known that the inclusion of infected classes in predator-prey models substantially change the dynamics of the original model. In particular, the inclusion of infected classes in the model can have a considerable impact on the population size of the predator-prey community~\cite{Hethcote-Wang-Han-Ma-TPB-2004, Koopmans-Wilbrink-Conyn-Natrop-Nat-Vennema-Lancet-2004}.}

{Lately, several works related to eco-epidemiological models have appeared in the literature. In~\cite{Chakraborty-Das-Haldar-Kar-2015}, the authors study the extinction and persistence of the disease in some eco-epidemiological systems; in~\cite{Bai-Xo-2018} the global stability of a delayed eco-epidemiological model with Holling type III functional response is addressed, and in~\cite{Purnomo-Darti-Suryanto-2017} the authors study an eco-epidemiological model with harvesting.}

{We note that the parameters in the eco-epidemiological models referred above are constant. On the other hand, to make models more consistent with reality it is seldom important to consider parameters that vary in time. Recently, several eco-epidemiological models with time varying parameters, particularly models with periodic coefficients have been studied~\cite{Dobson-QRV-1988, Friend-H-2002, Koopmans-Wilbrink-Conyn-Natrop-Nat-Vennema-Lancet-2004, Krebs-Blackwell-Scientific-Publishers-1978, Niu-Zhang-Teng-AMM-2011,Jesus-Silva-Vilarinho-preprint1-2020}. In the more general situation of nonautonomous models that are not necessarily periodic, threshold conditions for the extinction and persistence of the infected preys are obtained in~\cite{Niu-Zhang-Teng-AMM-2011} for a family of non-autonomous eco-epidemiological models with disease in the prey and no predation on uninfected preys. The results in that paper are generalised in~\cite{Jesus-Silva-Vilarinho-preprint2-2020} for a class of non-autonomous eco-epidemiological models that
include general functions corresponding to the predation on uninfected prey and also to the vital dynamics of uninfected prey and predator populations. Note that already in~\cite{Lu-Wang-Liu-DCDS-B-2018} a family of models that include predation on uninfected preys where considered, assuming that predation on uninfected prey is given by a bilinear functional response and also some particular form for the vital dynamics associated with uninfected preys and predators.}

{The approach in~\cite{Jesus-Silva-Vilarinho-preprint2-2020} is very different from the one in~\cite{Niu-Zhang-Teng-AMM-2011} and~\cite{Lu-Wang-Liu-DCDS-B-2018}: in~\cite{Jesus-Silva-Vilarinho-preprint2-2020} the uninfected subsystem corresponding to the dynamics of preys and predators in the absence of disease is not assumed to follow some special law but instead the hypothesis are on the stability of the referred uninfected subsystem. This approach allows the application of the results in~\cite{Jesus-Silva-Vilarinho-preprint2-2020} to eco-epidemiological models constructed from previously studied predator-prey model that satisfies the stability assumptions made.}

{In all the previous situations the models involved are continuous. In contrast, in this paper we consider a discrete version of the model in~\cite{Jesus-Silva-Vilarinho-preprint2-2020} obtained by applying Mickens discretization method. For the obtained model we derive a discrete version of the main result in that paper regarding the threshold dynamics of the model. We note that in~\cite{Hu-Teng-Jia-Zhang-ADE-2014} a discrete eco-epidemiological model was already studied. In contrast with our nonautonomous model, in that paper the model considered is autonomous and assumes no predation on uninfected preys. Additionally, in that paper the discretization method is very different from ours, resulting in a very different form for the equations obtained.}

{The structure of the present work is the following: in section~\ref{sec-1} we derive our model from the corresponding continuous model using Mickens nonstandard discretization scheme, establish our setting and some preliminary results; in section~\ref{sec2} we obtain our main result on extinction and persistence of the infective prey; finally, in section~\ref{sec3}, we consider some particular models that illustrate our results.}

\section{A general eco-epidemiological model with disease in {the} prey}\label{sec-1}

We consider the following non-autonomous eco-epidemiological model:
\begin{equation}\label{eq:principal}
\begin{cases}
S'=\Lambda(t)-\mu(t)S-a(t){f(S,I,P)P}-\beta(t)SI\\
I'=\beta(t)SI-\eta(t)g(S,I,P)I-c(t)I\\
P'=(r(t)-b(t)P)P+\gamma(t)a(t){f(S,I,P)P}+\theta(t)\eta(t)g(S,I,P)I
\end{cases},
\end{equation}
where $S$, $I$ and $P$ correspond, respectively, to the susceptible prey, infected prey and predator, $\Lambda(t)$ is the birth rate, $\mu(t)$ is the death rate of susceptible preys, $\beta(t)$ is the incidence rate of the disease, $\eta(t)$ is the predation rate of infected prey, $c(t)$ is the death rate in the infective class, $\gamma(t)$ is the
rate converting susceptible prey into predator (biomass transfer), $\theta(t)$ is the rate of converting infected prey into predator, $r(t)-b(t)P$ represent the vital dynamics of the predator populations, $a(t)f(S,I,P)$ is the predation of susceptible prey and  $\eta(t)g(S,I,P)$ is the predation of infected prey. It is assumed that only susceptible preys $S$ are capable of reproducing, i.e, the infected prey is removed by death (including natural and disease-related death) or by predation before having the possibility of reproducing.

The aim of this work is to discuss the uniform strong persistence and extinction of the infectives $I$ of the discrete counterpart of the system~\eqref{eq:principal}. A possible discretization of the above model, with stepsize $h$, derived by applying Mickens' nonstandard finite difference method~\cite{Mickens-JCAM-1999}, leads to the following set of equations:
\[
\begin{cases}
S(nh+h)-S(nh)
= h\Lambda(nh)-h\mu(nh)S(nh+h)\\
\quad\quad \quad \quad\quad \quad \quad\quad \quad \quad -ha(nh){f(S(nh+h),I(nh),P(nh))P(nh)}\\
\quad\quad \quad \quad\quad \quad \quad\quad \quad \quad -h\beta(nh)S(nh+h)I(nh)\\
I(nh+h)-I(nh)
= h\beta(nh)S(nh+h)I(nh)\\
\quad\quad \quad \quad\quad \quad \quad\quad \quad \quad -h\eta(nh)g(S(nh),I(nh),P(nh))I(nh+h)\\
\quad\quad \quad \quad\quad \quad \quad\quad \quad \quad  -hc(nh)I(nh+h)\\
P(nh+h)-P(nh)=h(r(nh)-b(nh)P(nh+h))P(nh)+h\gamma(nh)a(nh)\times\\
\quad\quad \quad \quad\quad \quad \quad\quad \quad \quad
\times {f(S(nh+h),I(nh),P(nh))P(nh)}\\
\quad\quad \quad \quad\quad \quad \quad\quad \quad \quad +h\theta(nh)\eta(nh)g(S(nh),I(nh),P(nh))I(nh+h)
\end{cases}.
\]
Using the notation $\xi_n=h\xi(nh)$ for $g=\Lambda, \mu, a, \beta, \eta, c, r, b$ and also $\zeta_n=\zeta(nh)$ for $\zeta=\gamma, \theta$, we obtain the following system of difference equations:
\begin{equation}\label{eq:principal-disc}
\begin{cases}
S_{n+1}-S_n=\Lambda_n-\mu_nS_{n+1}-a_n{f(S_{n+1},I_n,P_n)P_n}-\beta_nS_{n+1}I_n\\
I_{n+1}-I_n=\beta_nS_{n+1}I_n-\eta_ng(S_n,I_n,P_n)I_{n+1}-c_nI_{n+1}\\
P_{n+1}-P_n=(r_n-b_nP_{n+1})P_n+\gamma_na_n{f(S_{n+1},I_n,P_n)P_n}+\theta_n\eta_ng(S_n,I_n,P_n)I_{n+1}
\end{cases}.
\end{equation}

We will assume that
\begin{enumerate}[H$1$)]
\item \label{Cond-1} $(\Lambda_n)$, $(a_n)$, $(\beta_n)$, $(\eta_n)$, $(c_n)$, $(r_n)$, $(b_n)$, $(\gamma_n)$ and $(\theta_n)$ are bounded and nonnegative sequences and  $0<\mu_n \le c_n$;
\item \label{Cond-1a}$(\Lambda_n)$, $(r_n)$ and $(b_n)$ are bounded away from zero;
\item \label{cond-f} $f,g:(\R_0^+)^3\to \R$ are $C^1$ nonnegative; for fixed $x,z \ge 0$, $y\mapsto f(x,y,z)$ and $y \mapsto g(x,y,z)$ are nonincreasing; for fixed $y,z \ge 0$, $x \mapsto f(x,y,z)$ is nondecreasing and $x \mapsto g(x,y,z)$ is nonincreasing; for fixed $x,y \ge 0$, $z \mapsto g(x,y,z)$ is nonincreasing and $z \mapsto g(x,y,z)$ is nondecreasing;
\item \label{Cond-2} there is $\omega \in \N$ such that
    $$\limsup_{n \to +\infty} \prod_{k=n}^{n+\omega} \dfrac{1}{1+\mu_k} < 1.$$
\end{enumerate}

It follows from~H\ref{Cond-2}) that there are constants $K>0$ and $\theta \in ]0,1[$ such that
\begin{equation}\label{eq:bahavior-prod-Lambda:(1+d)}
\prod_{k=m}^{n-1} \dfrac{1}{1+\mu_k} < K \theta^{n-m},
\end{equation}
for any $m,n \in \N_0$ with $n>m$.
{
\begin{enumerate}[H$1$)]
\setcounter{enumi}{4}
\item \label{Cond-aditional} Given $p \in \N$ there is a unique solution $((S_n,I_n,P_n))_{n \ge p}$ of
    system~\eqref{eq:principal-disc} with initial condition $(S_p,I_p,P_p) \in (\R_0^+)^3$.
\item \label{Cond-aditional2} Any solution of system~\eqref{eq:principal-disc} with nonnegative (resp. positive) initial condition, $(S_q,I_q,P_q)$  is nonnegative (resp. positive) for all $n \ge q$.
\end{enumerate}
}
Note that {when $f(S_{n+1},I_n,P_n)=S_{n+1}$ and $g(S_n,I_n,P_n)=P_n$ in~\eqref{eq:principal-disc}, the equation can be rewritten in explicit form:
\begin{equation}\label{eq:principal-disc-explicit}
\begin{cases}
S_{n+1}=\dfrac{\Theta_n}{\Psi_n}\\[4mm]
I_{n+1}=\dfrac{\beta_n\Theta_n+\Psi_n}{\Psi_n\Phi_n}I_n\\[4mm]
P_{n+1}=\dfrac{(1+r_n)\Psi_n\Phi_n+\gamma_na_n\Theta_n\Phi_n+\theta_n\eta_n(\Psi_n+\beta_n\Theta_n)I_n}{\Psi_n\Phi_n(1+b_nP_n)}P_n
\end{cases},
\end{equation}
where $\Psi_n=1+\mu_n+\beta_nI_n+a_nP_n$, $\Phi_n=1+\eta_nP_n+c_n$ and $\Theta_n=\Lambda_n+S_n$.
From~\eqref{eq:principal-disc-explicit}, we conclude {that when $f(S_{n+1},I_n,P_n)=S_{n+1}$ system~\eqref{eq:principal-disc} is well defined and H\ref{Cond-aditional2}) holds. Let us introduce the notation $f^\ell=\inf f_n$ and $f^u=\sup f_n$.

To proceed, we need to consider two auxiliary equations. The first one corresponds to the dynamics of preys in the absence of infected preys and predators:
\[
s_{n+1}-s_n=\Lambda_n-\mu_n s_{n+1}.
\]
Rearranging terms, the equation above becomes:
\begin{equation}\label{eq:aux-syst-S}
s_{n+1}=\frac{\Lambda_n}{1+\mu_n}+\frac{s_n}{1+\mu_n}.
\end{equation}
We have the following lemma that was essentially proved in~\cite{Mateus-IJDE-2016}:
\begin{lemma}\label{lemma:aux-S}
We have the following:
\begin{enumerate}[i)]
\item \label{Cond-C1-aux}
The solution of equation~\eqref{eq:aux-syst-S} with $s_p=0$ is the identically null sequence;
\item \label{Cond-C1a-aux}
All solutions $(s_n)$ of equation~\eqref{eq:aux-syst-S} with initial condition $s_0 > 0$ are positive for all $n \in \N$;
\item \label{Cond-C1b-aux}
Given a solution $(s_n)$ of equation~\eqref{eq:aux-syst-S} with initial condition $s_0 \in [\Lambda^\ell/\mu^u, \Lambda^u/\mu^\ell]$ we have
    $$\dfrac{\Lambda^\ell}{\mu^u} \le s_n \le \dfrac{\Lambda^u}{\mu^\ell}$$ for all $n \in \N$;
\item \label{Cond-C3-aux} Each fixed solution $(s_n)$ of~\eqref{eq:aux-syst-S} with initial condition $s_0 \ge 0$ is bounded and globally uniformly attractive on $[0,+\infty)$;
\item \label{cond-3-aux} There is a constant $D>0$ such that if $(s_n)$ is a solution of~\eqref{eq:aux-syst-S} and $(\tilde s_n)$ is a solution of the system
\begin{equation}\label{eq:sist-aux-f}
s_{n+1}=\dfrac{\Lambda_n +s_n+\phi_n}{1+ \mu_n}
\end{equation}
with $\tilde s_0=s_0$ then
\[
\sup_{n \ge 0} |\tilde s_n - s_n| \le D \ \sup_{n \ge 0} |\phi_n|.
\]
\item \label{cond-4-aux} There is a constant $E>0$ such that if $(s_n)$ is a solution of~\eqref{eq:aux-syst-S} and $(\tilde s_n)$ is a solution of the system
\begin{equation}\label{eq:sist-aux-f-2}
s_{n+1}=\dfrac{\Lambda_n +s_n}{1+ \mu_n+\psi_n}
\end{equation}
with $\tilde s_0=s_0$ then there is $N_1$ sufficiently large such that
\[
\sup_{n \ge N_1} |\tilde s_n - s_n| \le E \ \sup_{n \ge N_1} |\psi_n|.
\]
\end{enumerate}
\end{lemma}

\begin{proof}
Properties~\ref{Cond-C1-aux}) to~\ref{cond-3-aux}) follow from Lemma 1 in~\cite{Mateus-IJDE-2016}. To prove~\ref{cond-4-aux}),
let $(s_n)$ be a solution of~\eqref{eq:aux-syst-S} and $(\tilde s_n)$ be a solution of~\eqref{eq:sist-aux-f-2} with $\tilde s_0=s_0$. By~\eqref{eq:aux-syst-S} and~\eqref{eq:sist-aux-f-2}, we have
\[
(\tilde s_{n+1}-s_{n+1})(1+\mu_n)=\tilde s_n-s_n-g_n \tilde s_{n+1}
\]
Therefore, letting $w_n=|\tilde s_n-s_n|$, we have
\[
w_{n+1}(1+\mu_n)\le w_n + |g_n| \tilde s_{n+1}
\]
and thus
\[
w_{n+1} \le \frac{w_n}{1+\mu_n} + \frac{|\psi_n|\tilde s_{n+1}}{1+\mu_n}
\]
Fix $\eps>0$. By~\ref{Cond-C1b-aux}) and~\ref{Cond-C3-aux}) we get, for $n$ sufficiently large, say $n \ge N_1$,
\[
w_{n+1} \le \frac{w_n}{1+\mu_n} + \frac{|\psi_n|}{1+\mu_n}\left[\frac{\Lambda^u}{\mu^\ell+\psi^\ell}+\eps\right]
\]
and thus, for $n \ge N_1$,
\[
\begin{split}
w_n
& \le  \left[\frac{\Lambda^u}{\mu^\ell+\psi^\ell}+\eps\right] \sum_{m=0}^{n-1} |\psi_m| \left(\prod_{k=m}^{n-1}\dfrac{1}{1+\mu_k}\right)\\
& \le  \left[ \frac{\Lambda^u}{\mu^\ell+\psi^\ell}+\eps \right]  \ \sup_{n \ge 0} |\psi_n| \, K \sum_{m=0}^{n-1} \theta^{n-m}\\
& \le  \left[ \frac{\Lambda^u}{\mu^\ell+\psi^\ell}+\eps \right] \, \frac{K\theta}{1-\theta} \, \sup_{n \ge 0}|\psi_n|.
\end{split}
\]
Defining $E= K\theta\left[\Lambda^u/(\mu^\ell+\psi^\ell)+\eps \right] /(1-\theta)$, we get
\[
\sup_{n \ge N_1} |\tilde s_n - s_n| = \sup_{n \ge N_1} w_n \le w_n \le E \ \sup_{n \ge N_1} |\psi_n|,
\]
and the result follows.
\end{proof}

We also need to consider the equation:
\[
y_{n+1}-y_n=(r_n-b_ny_{n+1})y_n.
\]
Rearranging terms, we get:
\begin{equation}\label{eq:aux-syst-P}
y_{n+1}=\frac{r_ny_n+y_n}{1+b_ny_n}
\end{equation}
The following lemma holds.
\begin{lemma}\label{lemma:aux-P}
We have the following:
\begin{enumerate}[i)]
\item \label{Cond-C1-aux-y}
The solution of equation~\eqref{eq:aux-syst-P} with $y_p=0$ is the identically null sequence;
\item \label{Cond-C1a-aux-y}
All solutions $(y_n)$ of equation~\eqref{eq:aux-syst-P} with initial condition $y_0 > 0$ are positive for all $n \in \N$;
\item \label{Cond-C1b-aux-y}
Given a solution $(y_n)$ of equation~\eqref{eq:aux-syst-P} with initial condition $y_0 \in [r^\ell/b^u,r^u/b^\ell]$ we have
    $$\frac{r^\ell}{b^u} \le y_n \le \dfrac{r^u}{b^\ell}$$ for all $n \in \N$;
\item \label{Cond-C3-aux-y} Each fixed solution $(y_n)$ of~\eqref{eq:aux-syst-P} with initial condition $y_0 > 0$ is bounded and globally uniformly attractive on $]0,+\infty)$;
\item \label{cond-3-aux-y} There is a constant $E>0$ such that, if $(y_n)$ is a solution of~\eqref{eq:aux-syst-P} and $(\tilde y_n)$ is a solution of the system
\begin{equation}\label{eq:sist-aux-g-y}
y_{n+1}=\frac{r_ny_n+y_n}{1+(b_n+g_n)y_n}, \quad n=0,1,\ldots
\end{equation}
with $\tilde y_0=y_0$ then there is $N_1$ sufficiently large such that
\[
\sup_{n \ge N_1} |\tilde y_n - y_n| \le E \ \sup_{n \ge N_1} |g_n|.
\]
\item \label{cond-4-aux-y} There is a constant $G>0$ such that, if $(y_n)$ is a solution of~\eqref{eq:aux-syst-P} and $(\tilde y_n)$ is a solution of the system
\begin{equation}\label{eq:sist-aux-g-y2}
y_{n+1}=\frac{(r_n+h_n)y_n+y_n}{1+b_ny_n}, \quad n=0,1,\ldots
\end{equation}
with $\tilde y_0=y_0$ then there is $N_2$ sufficiently large such that
\[
\sup_{n \ge N_2} |\tilde y_n - y_n| \le G \ \sup_{n \ge N_2} |h_n|.
\]
\end{enumerate}
\end{lemma}

\begin{proof}
With the change of variable $w_n=1/y_n$, equation~\eqref{eq:aux-syst-P} becomes
\[
w_{n+1}=\frac{b_n}{r_n+1}+\frac{w_n}{r_n+1},
\]
equation~\eqref{eq:sist-aux-g-y} becomes
$$w_{n+1}=\frac{w_n+b_n+ g_n}{r_n+1}.$$
and equation~\eqref{eq:sist-aux-g-y2} becomes
$$w_{n+1}=\frac{w_n+b_n}{1+r_n+h_n}.$$
Using Lemma~\ref{lemma:aux-S}, we obtain~\ref{Cond-C1a-aux-y}) to~\ref{cond-4-aux-y}). Property~\ref{Cond-C1-aux-y}) is immediate.
\end{proof}

We must assume the following:
\begin{enumerate}[H$1$)]
\setcounter{enumi}{6}
\item \label{Cond-7a} Each solution of~\eqref{eq:principal-disc} with positive initial condition is bounded and there is a bounded region $\mathcal R$ that contains the $\omega$-limit of all solutions of~\eqref{eq:principal-disc} with positive initial conditions.
\end{enumerate}
Notice in particular that condition~H\ref{Cond-7a}) implies that there is $L>0$ such that, for each solution $(S_n,I_n,P_n)$ we have
\begin{equation}\label{eq:bound}
\limsup_{t \to +\infty} \, (S_n+I_n+P_n) < L.
\end{equation}

The next lemma shows that, when $g(S,I,P)=g_0(S,I)P$, there is an invariant region that attracts all orbits of system~\eqref{eq:principal-disc}.
{\begin{lemma}\label{lemma:region}
Assume that $g(S,I,P)=g_0(S,I)P$. Then, there is $L>0$ such that, for any solution $(S_n,I_n,P_n)$ of~\eqref{eq:principal-disc},  with nonnegative initial conditions, there is $T \in \N$ such that
\[
S_n+I_n+P_n \le L \quad \text{for} \quad n \ge T.
\]
\end{lemma}}
\begin{proof}
Let $(S_n,I_n,P_n)$ be a solution of~\eqref{eq:principal-disc} with nonnegative initial conditions $S_q=s_q$, $I_q=i_q$ and $P_q=p_q$.
Adding the first two equations in~\eqref{eq:principal-disc} and writing $N_n=S_n+I_n$, we get
\[
\begin{split}
N_{n+1}-N_n
& =\Lambda_n-\mu_nS_{n+1}-c_nI_{n+1}-a_nf(S_{n+1},I_n,P_n)P_n\\
& \phantom{=} \ -\eta_ng_0(S_{n+1},I_n)P_nI_{n+1} \\
& \le \Lambda_n - \mu_nN_{n+1},
\end{split}
\]
since $\mu_n=\min\{\mu_n,c_n\}$. Thus
\[
N_{n+1} \le \dfrac{\Lambda_n}{1+\mu_n}+\dfrac{N_n}{1+\mu_n}.
\]
By~\ref{Cond-C1b-aux}) and~\ref{Cond-C3-aux}) in Lemma~\ref{lemma:aux-S}, we conclude that, for any given $\eps>0$, we have $S_n+I_n = N_n \le s_n \le \Lambda^u/\mu^\ell+\eps$, where $s_n$ is a solution of~\eqref{eq:aux-syst-S} with initial condition $s_q=N_q$, for $n$ sufficiently large, say  $n \ge N_1$.

By the third equation in~\eqref{eq:principal-disc} we obtain
\[
\begin{split}
P_{n+1}
& =\dfrac{P_n+r_nP_n+\gamma_n a_nf(S_{n+1},I_n,P_n)P_n+\theta_n\eta_ng_0(S_{n+1},I_n)P_nI_{n+1}}{1+b_nP_n}\\
& \le \dfrac{\left[r_n+\gamma_n a_n f(\Lambda^u/\mu^\ell+\eps,0,0)+\theta_n\eta_n g_0(0,\Lambda^u/\mu^\ell+\eps)(\Lambda^u/\mu^\ell+\eps)\right]P_n+P_n}{1+b_nP_n}
\end{split}
\]
for $n \ge N_1$. By~\ref{Cond-C1b-aux-y}) and~\ref{Cond-C3-aux-y}) in Lemma~\ref{lemma:aux-P}, we conclude that, for any given $\delta>0$, there is $N_2 \ge N_1$ such that, for all $n \ge N_2$
\[
\begin{split}
P_n
& \le \frac{\sup_{n \ge q}\left\{r_n+\gamma_n a_n f(\Lambda^u/\mu^\ell+\eps,0,0)+\theta_n\eta_n g_0(0,\Lambda^u/\mu^\ell+\eps)(\Lambda^u/\mu^\ell+\eps)\right\}}{b^\ell}+\delta\\
& \le \frac{r^u+\gamma^u a^u f(\Lambda^u/\mu^\ell+\eps,0,0)+\theta^u \eta^u g_0(0,\Lambda^u/\mu^\ell+\eps)(\Lambda^u/\mu^\ell+\eps)}{b^\ell}+\delta.
\end{split}
\]
Thus
\[
S_n+I_n+P_n \le \frac{\Lambda^u}{\mu^\ell}+\eps+\frac{r^u+\gamma^u a^u f(\Lambda^u/\mu^\ell+\eps,0,0)+\theta^u \eta^u g_0(0,\Lambda^u/\mu^\ell+\eps)(\Lambda^u/\mu^\ell+\eps)}{b^\ell}+\delta,
\]
and the result follows.
\end{proof}

To formulate our next assumption we need to consider the system
\begin{equation}\label{eq:aux2-disc}
\begin{cases}
x_{n+1}-x_n=\Lambda_n-\mu_n x_{n+1}-a_n{f(x_{n+1},0,z_n)z_n}\\
z_{n+1}-z_n=(r_n-b_nz_{n+1})z_n+\gamma_na_n{f(x_{n+1},0,z_n)z_n}
\end{cases}.
\end{equation}
which corresponds to the dynamics of the susceptible preys and the predators in the absence of infected preys.
We also need to consider the two families of auxiliary systems:
\begin{equation}\label{eq:aux2a-disc}
\begin{cases}
x_{n+1}-x_n=\Lambda_n-\mu_n x_{n+1}-a_n{f(x_{n+1},0,z_n)z_n}-\eps x_n\\
z_{n+1}-z_n=(r_n-b_nz_{n+1})z_n+\gamma_na_n{f(x_{n+1},\eps,z_n)z_n}
\end{cases}
\end{equation}
and
\begin{equation}\label{eq:aux2b-disc}
\begin{cases}
x_{n+1}-x_n=\Lambda_n-\mu_n x_{n+1}-a_n{f(x_{n+1},\eps,z_n)z_n}\\
z_{n+1}-z_n=(r_n-b_nz_{n+1})z_n+\gamma_na_n{f(x_{n+1},0,z_n)z_n}+\theta_n\eta_ng(x_{n+1},0,z_n)\eps
\end{cases}.
\end{equation}
We make the following assumptions concerning systems~\eqref{eq:aux2a-disc} and~\eqref{eq:aux2b-disc}.
\begin{enumerate}[H$1$)]
\setcounter{enumi}{7}
\item \label{Cond-4}
There is a family of nonnegative solutions $(x^*_{1,\eps,n},z^*_{1,\eps,n})$ of system~\eqref{eq:aux2a-disc}, one for each $\eps>0$ sufficiently small, such that each solution in the family is globally asymptotically stable in a set containing $\{(x,y) \in (R_0^+)^2: x,z>0\}$ and the function $\eps \mapsto (x^*_{1,\eps,n},z^*_{1,\eps,n})$ is continuous.
\item \label{Cond-5}
There is a family of nonegative solutions $(x^*_{2,\eps,n},z^*_{2,\eps,n})$ of system~\eqref{eq:aux2b-disc}, one for each $\eps>0$ sufficiently small, such that each solution in the family is globally asymptotically stable in a set containing $\{(x,y) \in (R_0^+)^2: x,z>0\}$ and the function $\eps \mapsto (x^*_{2,\eps,n},z^*_{2,\eps,n})$ is continuous.
\end{enumerate}
We denote the element of the family of solutions in~{H8)}  (or~{H9)}) with $\eps=0$, by $(x^*_n,z^*_n)$. For each solution $(x^*_n,z^*_n)$ of~\eqref{eq:aux2a-disc} with $\eps=0$ and initial conditions $(x_0,y_0)$ with $x_0>0$ and $z_0>0$, and each $\lambda \in \N$, define the number
\begin{equation}\label{eq:liminf-threshold}
\cR^\ell(\lambda) = \liminf_{n \to \, +\infty} \prod_{k=n}^{n+\lambda} \dfrac{1+\beta_kx^*_{k+1}}{1+c_k+\eta_k g(x^*_k,0,z^*_k)}
\end{equation}
and for each solution $(s_n^*)$ of~\eqref{eq:aux-syst-S} with $s_0>0$, each solution $(y_n^*)$ of~\eqref{eq:aux-syst-P} with $y_0>0$ and each $\lambda \in \N$, define the number
\begin{equation}\label{eq:limsup-threshold}
\cR^u(\lambda) = \limsup_{n \to \, +\infty} \prod_{k=n}^{n+\lambda} \dfrac{1+\beta_ks^*_{k+1}}{1+c_k+\eta_k  g(s^*_k,0,y^*_k)}
\end{equation}
These numbers will be useful in obtaining conditions for permanence and extinction and, in some sense, play the role of upper and lower bounds for the basic reproductive number in this general context. In the following lemma we prove that the numbers above are independent of the particular positive solutions of~\eqref{eq:aux-syst-S}, ~\eqref{eq:aux-syst-P} and~\eqref{eq:aux2a-disc} considered.

\begin{lemma} \label{lema:indep}
The numbers $\cR^\ell(\lambda)$ and $\cR^u(\lambda)$ are independent of the particular solutions $(s_n^*)$ of~\eqref{eq:aux-syst-S} with $s_0>0$,  $(y_n^*)$ of~\eqref{eq:aux-syst-P} with $y_0>0$ and $(x_n^*,z_n^*)$ of~\eqref{eq:aux2-disc} with $x_0>0$ and $z_0>0$.
\end{lemma}

\begin{proof} Write $\cR^\ell(\lambda,x,z)$ for the number in~\eqref{eq:liminf-threshold} corresponding to the solution $(x,z)=(x_n^*,z_n^*)_{n \in \N}$ of~\eqref{eq:aux2-disc} with $x_0>0$ and $z_0>0$.

Let $(x_1^*,z_1^*)=(x^*_{1,n},z^*_{1,n})_{n \in \N}$ and $(x_2^*,z_2^*)=(x^*_{2,n},z^*_{2,n})_{n \in \N}$ be distinct solutions of~\eqref{eq:aux2-disc} with $x_{1,0}>0,z_{1,0}>0,x_{2,0}>0$ and $z_{2,0}>0$.

Let $\delta>0$ be sufficiently small. By assumptions {H8)} (or {H9)}), for $k \ge N$ (where $N \in \N$) sufficiently large, we have
$$x^*_{1,k} -\delta \le x_{2,k}^* \le x_{1,k}^*+ \delta \quad \text{ and } \quad z^*_{1,k} -\delta \le z_{2,k}^* \le z_{1,k}^*+ \delta.$$
Additionally, by H\ref{cond-f}), there is $c>0$ such that, for sufficiently large $k$,
$$|g(x^*_{1,k},0,z^*_{1,k})-g(x^*_{1,k},0,z^*_{2,k}-\delta)|\le c |z^*_{1,k}-z^*_{2,k}+\delta| \le 2c\delta $$
and
$$|g(x^*_{2,k},0,z^*_{1,k})-g(x^*_{1,k},0,z^*_{1,k})|\le c |z^*_{1,k}-z^*_{2,k}| \le c\delta.$$
Thus, for $n \ge N$
\begin{equation}\label{eq:estim-prod-le0}
\begin{split}
& \quad \prod_{k=n}^{n+\lambda} \dfrac{1+\beta_kx^*_{2,k+1}}{1+c_k+\eta_kg(x^*_{2,k},0,z^*_{2,k})}\\
& \le \prod_{k=n}^{n+\lambda} \dfrac{1+\beta_kx^*_{1,k+1}+\delta\beta_k}{1+c_k+\eta_kg(x^*_{2,k},0,z^*_{1,k}-\delta)}\\
& \le \prod_{k=n}^{n+\lambda} \dfrac{1+\beta_kx^*_{1,k+1}+\delta\beta_k}{1+c_k+\eta_kg(x^*_{1,k},0,z^*_{1,k})}\dfrac{1+c_k+\eta_kg(x^*_{1,k},0,z^*_{1,k})}{1+c_k+\eta_kg(x^*_{2,k},0,z^*_{1,k}-\delta)}\\
& = \prod_{k=n}^{n+\lambda} \left(\dfrac{1+\beta_kx^*_{1,k+1}}{1+c_k+\eta_kg(x^*_{1,k},0,z^*_{1,k})}+\dfrac{\delta\beta_k}{1+c_k+\eta_kg(x^*_{1,k},0,z^*_{1,k})}\right)\times\\
& \phantom{==} \times
\left(1+\dfrac{3c\delta\eta_k}{1+c_k+\eta_kg(x^*_{2,k},0,z^*_{2,k}-\delta)}\right)\\
& \le (1+\delta B)^\lbd \prod_{k=n}^{n+\lambda} \left(\dfrac{1+\beta_kx^*_{1,k+1}}{1+c_k+\eta_kg(x^*_{1,k},0,z^*_{1,k})}+\delta A\right)
\\
& \le (1+\delta B)^\lbd \left(\prod_{k=n}^{n+\lambda} \dfrac{1+\beta_kx^*_{1,k+1}}{1+c_k+\eta_kg(x^*_{1,k},0,z^*_{1,k})}
+  \sum_{j=1}^{\lambda+1} \binom{\lambda+1}{j}\delta^j  C^{\lambda+1-j} A^j\right),
\end{split}
\end{equation}
where
$$A=\dfrac{\beta^u}{1+c^\ell+\eta^\ell (g(x^*_{1,k},0,z^*_{1,k}))^\ell}, \quad B=\dfrac{2c\eta^u}{1+c^\ell+\eta^\ell (g(x^*_{1,k},0,z^*_{2,k}-\delta))^\ell}$$
and
$$C=\dfrac{1+\beta^u (x^*_1)^u}{1+c^\ell+\eta^\ell (g(x^*_{1,k},0,z^*_{1,k}))^\ell}.$$

By~\eqref{eq:estim-prod-le0}, we conclude that
$$\cR^\ell(\lambda,x^*_2,z^*_2) \le (1+\delta B)^\lbd \left(\cR^\ell(\lambda,x^*_1,z^*_1)
+  \sum_{j=1}^{\lambda+1} \binom{\lambda+1}{j}\delta^j  C^{\lambda+1-j} A^j\right).$$

By the arbitrariness of $\delta>0$, we conclude that $\cR^\ell(\lambda,x^*_2,z^*_2) \le \cR^\ell(\lambda,x^*_1,z^*_1)$ and, interchanging the roles of $(x^*_1,z^*_1)$ and $(x^*_2,z^*_2)$ it is immediate that
$\cR^\ell(\lambda,x^*_2,z^*_2) \ge \cR^\ell(\lambda,x^*_1,z^*_1)$. Thus $\cR^\ell(\lambda,x^*_2,z^*_2) = \cR^\ell(\lambda,x^*_1,z^*_1)$.

Now write $\cR^u(\lambda,s,y)$ for the number in~\eqref{eq:limsup-threshold} corresponding to the solutions $s=(s_n^*)$ of~\eqref{eq:aux-syst-S} with $s_0>0$ and  $y=(y_n^*)$ of~\eqref{eq:aux-syst-P} with $y_0>0$.

Let again $\delta>0$ be sufficiently small. Additionally, let $s_1^*=(s^*_{1,n})$ and $s_2^*=(s^*_{2,n})$ be distinct solutions of~\eqref{eq:aux-syst-S} and $y_1^*=(y^*_{1,n})$ and $y_2^*=(y^*_{2,n})$ be distinct solutions of~\eqref{eq:aux-syst-P}. By \ref{Cond-C3-aux}) in Lemma~\ref{lemma:aux-S} and \ref{Cond-C3-aux-y}) in Lemma~\ref{lemma:aux-P}, we have
$$s^*_{1,k} -\delta \le s_{2,k}^* \le s_{2,k}^*+ \delta \quad \text{ and } \quad y^*_{1,k} -\delta \le y_{2,k}^* \le y_{1,k}^*+ \delta$$
for $k \ge N$ sufficiently large. There is $c>0$ such that
$$|g(s^*_{1,k},0,y^*_{1,k})-g(s^*_{1,k},0,y^*_{2,k}-\delta)|\le c |y^*_{1,k}-y^*_{2,k}+\delta| \le 2c\delta$$
and
$$|g(s^*_{2,k},0,y^*_{1,k})-g(s^*_{1,k},0,y^*_{1,k})|\le c |s^*_{2,k}-s^*_{1,k}| \le c\delta.$$
Therefore
\begin{equation}\label{eq:estim-prod-le}
\begin{split}
& \quad \prod_{k=n}^{n+\lambda} \dfrac{1+\beta_ks^*_{2,k+1}}{1+c_k+\eta_kg(s_{2,k}^*,0,y^*_{2,k})}\\
& \le \prod_{k=n}^{n+\lambda} \dfrac{1+\beta_ks^*_{1,k+1}+\delta\beta_k}{1+c_k+\eta_kg(s_{2,k}^*,0,y^*_{1,k}-\delta)}\\
& \le \prod_{k=n}^{n+\lambda} \dfrac{1+\beta_ks^*_{1,k+1}+\delta\beta_k}{1+c_k+\eta_kg(s_{1,k}^*,0,y^*_{1,k})}\dfrac{1+c_k+\eta_kg(s_{1,k}^*,0,y^*_{1,k})}{1+c_k+\eta_kg(s_{2,k}^*,0,y^*_{1,k}-\delta)}\\
& \le \prod_{k=n}^{n+\lambda} \left(\dfrac{1+\beta_ks^*_{1,k+1}}{1+c_k+\eta_kg(s_{1,k}^*,0,y^*_{1,k})}+\dfrac{\delta\beta_k}{1+c_k+\eta_kg(s_{1,k}^*,0,y^*_{1,k})}\right)\times\\
& \phantom{=} \times\left(1+\dfrac{3c\delta\eta_k}{1+c_k+\eta_kg(s_{2,k}^*,0,y^*_{1,k}-\delta)}\right)\\
& \le (1+\delta B)^\lbd \prod_{k=n}^{n+\lambda} \left(\dfrac{1+\beta_ks^*_{1,k+1}}{1+c_k+\eta_kg(s_{1,k}^*,0,y^*_{1,k})}+\delta A\right)
\\
& \le (1+\delta B)^\lbd \left(\prod_{k=n}^{n+\lambda} \dfrac{1+\beta_ks^*_{1,k+1}}{1+c_k+\eta_kg(s_{1,k}^*,0,y^*_{1,k})}
+  \sum_{j=1}^{\lambda+1} \binom{\lambda+1}{j}\delta^j  C^{\lambda+1-j} A^j\right),
\end{split}
\end{equation}
for $n \ge N$, where
$$A=\dfrac{\beta^u}{1+c^\ell+\eta^\ell (g(s_{1,k}^*,0,y^*_{1,k}))^\ell}, \quad B=\dfrac{2c\eta^u}{1+c^\ell+\eta^\ell (g(s_{2,k}^*,0,y^*_{1,k}-\delta))^\ell}$$
and
$$C=\dfrac{1+\beta^u (s^*_1)^u}{1+c^\ell+\eta^\ell (g(s_{1,k}^*,0,y^*_{1,k}))^\ell}.$$

By~\eqref{eq:estim-prod-le}, we conclude that
$$\cR^\ell(\lambda,s^*_2,y^*_2) \le (1+\delta B)^\lbd \left(\cR^\ell(\lambda,s^*_1,y^*_1)
+  \sum_{j=1}^{\lambda+1} \binom{\lambda+1}{j}\delta^j  C^{\lambda+1-j} A^j,\right)$$

By the arbitrariness of $\eps>0$, we conclude that $\cR^\ell(\lambda,s^*_2,y^*_2) \le \cR^\ell(\lambda,s^*_1,y^*_1)$ and, interchanging the roles of $(s^*_1,y^*_1)$ and $(s^*_2,y^*_2)$ it is immediate that
$\cR^\ell(\lambda,s^*_2,y^*_2) \ge \cR^\ell(\lambda,s^*_1,y^*_1)$. Thus $\cR^\ell(\lambda,s^*_2,y^*_2) = \cR^\ell(\lambda,s^*_1,y^*_1)$.

The result is proved.
\end{proof}

{\section{Extinction and strong persistence}\label{sec2}}

In this section we establish our main results on extinction and persistence. To obtain our result on extinction we must make some additional assumptions on the function $g$. In spite of this, it is easy to see that the usual growth rates still fulfill these assumptions.

\begin{theorem}\label{teo:Main-ext}
Assume that $g(S+I,0,P)\le g(S,I,P)$. If there is $\lambda \in \N$ such that $\cR^u(\lambda)<1$ then the infectives $(I_n)$ go to extinction in system~\eqref{eq:principal-disc}. Furthermore, if $a \equiv 0$ and $g(S,I,P)=g_0(S,I)P$, any disease-free solution $(s^*_n,0,y^*_n)$ of~\eqref{eq:principal-disc}, where $(s^*_n)$ is a solution of~\eqref{eq:aux-syst-S} and $(y^*_n)$ is a solution of~\eqref{eq:aux-syst-P}, is globally asymptotically attractive.
\end{theorem}

\begin{proof}
Since $\cR^u(\lbd)<1$, given $\delta_1>0$ sufficiently small, there are $\delta_0>0$ and $N \in \N$ such that
\begin{equation}\label{eq:condi-ext}
\prod_{k=n}^{n+\lambda} \dfrac{1+\beta_k(s^*_{k+1}+\delta)}{1+c_k+\eta_kg(s^*_k+\delta,0,y^*_k-\delta))}<1-\delta_1,
\end{equation}
for $n \ge N$ and all positive $\delta\le \delta_0$. Let $N_n=S_n+I_n$. Since $\mu_n \le c_n$, by the first two equations in~\eqref{eq:principal-disc}, we conclude that
  $$N_{n+1}-N_n \le \Lambda_n-\mu_n N_{n+1} \quad \Leftrightarrow \quad N_{n+1} \le \frac{\Lambda_n}{1+\mu_n}+\frac{N_n}{1+\mu_n} $$
and thus $S_n, I_n \le N_n \le s_n$, where $(s_n)$ is any solution of~\eqref{eq:aux-syst-S} with $s_0=S_0$. By~\ref{Cond-C3-aux}) in Lemma~\ref{lemma:aux-S} we have $|s_n-s_n^*|\le \delta_0$ for sufficiently large $n$, say $n\ge N_1 \ge N$. Thus
$$S_n, I_n \le S_n+I_n=N_n \le s_n \le s_n^*+\delta_0,$$
for $n\ge N_1$.

By the third equation in~\eqref{eq:principal-disc}, we conclude that
  $$P_{n+1}-P_n\ge (r_n-b_nP_{n+1})P_n \quad \Leftrightarrow \quad P_{n+1}\ge \frac{r_nP_n+P_n}{1+b_nP_n}$$
and thus $P_n\ge y_n$, where $(y_n)$ is any solution of~\eqref{eq:aux-syst-S} with $y_0=P_0$. By~\ref{Cond-C3-aux}) in Lemma~\ref{lemma:aux-P} we have $|y_n-y_n^*|\le \delta_0$ for sufficiently large $n$, say $n\ge N_2 \ge N_1$. Thus
$$P_n\ge y_n \ge y_n^*-\delta_0,$$
for $n\ge N_2$.
 Using our hypothesis, by the second equation in~\eqref{eq:principal-disc} and~\eqref{eq:condi-ext},
\[
\begin{split}
I_{n+1}
& =\dfrac{\beta_nS_{n+1}I_n+I_n}{1+\eta_ng(S_n,I_n,P_n)+c_n}\\
& \le\dfrac{\beta_nS_{n+1}I_n+I_n}{1+\eta_ng(S_n+I_n,0,P_n)+c_n}\\
& \le \dfrac{\beta_n(s^*_{n+1}+\delta_0)+1}{1+c_n+\eta_ng(s_n^*+\delta_0,0,y^*_n-\delta_0)}\, I_n\\
&<(1-\delta_1)\,I_{n-\lambda-1}\\
&<\cdots<(1-\delta_1)^{\lfloor n/(\lambda+1)\rfloor} \, I_{n-\lfloor n/(\lambda+1)\rfloor(\lambda+1)}\\
& \le d \left((1-\delta_1)^{1/(\lambda+1)}\right)^n,
\end{split}
\]
for $n \ge N_2$, where $\displaystyle d=\max_{j=0,\ldots,\lambda} I_j$. We conclude that $I_n\to 0$ as $n\to +\infty$ and we have extinction of the infectives.

Assume now that $a \equiv 0$ and $g(S,I,P)=g_0(S,I)P$, let $((S_n,I_n,P_n))$ be any solution of~\eqref{eq:principal-disc} and consider the sequence $((s^*_n,0,y_n^*))$, where $(s^*_n)$ is a solution of~\eqref{eq:aux-syst-S} and $(y^*_n)$ is a solution of~\eqref{eq:aux-syst-P}.

Since $I_n\to 0$ as $n \to +\infty$, given $\delta>0$ there is $T\in\N$ such that $I_n<\delta$ for $n \ge T$. Letting $U_n=S_n-s_n^*$, we have, by the first equation in~\eqref{eq:principal-disc},
 $$U_{n+1}-U_n=-\mu_n U_{n+1} -\beta_nS_{n+1}I_n,$$
for $n \ge T$. Thus, by~\ref{Cond-C3-aux-y}) in Lemma~\ref{lemma:aux-S} and by Lemma~\ref{lemma:region}, we have
 $$-\beta^u L \delta<(1+\mu_n)U_{n+1}-U_n < 0$$
for $n$ sufficiently large.

We get, for $\delta>0$ sufficiently small
\[
\begin{split}
U_{n+1}
& > -\dfrac{\beta^u L \delta}{1+\mu_n}+\dfrac{1}{1+\mu_n}U_n\\
& > -\dfrac{\beta^u L \delta}{1+\mu_n}+\dfrac{1}{1+\mu_n}\left(-\dfrac{\beta^u K \delta}{1+\mu_{n-1}}+\dfrac{1}{1+\mu_{n-1}}U_{n-1}\right)\\
& > \cdots\\
& > \left(\prod_{m=0}^{n-1}\dfrac{1}{1+\mu_m}\right)U_0 - \sum_{m=0}^{n-1} (\beta^u L \delta)^{m+1} \left(\prod_{k=m}^{n-1}\dfrac{1}{1+\mu_k}\right)\\
& > \left(\prod_{m=0}^{n-1}\dfrac{1}{1+\mu_m}\right)U_0 - \delta \beta^u L \sum_{m=0}^{n-1} K \theta^{n-m}\\
& > \left(\prod_{m=0}^{n-1}\dfrac{1}{1+\mu_m}\right)U_0 - \dfrac{\beta^u LK\theta}{1-\theta}\,\delta\\
& > - \dfrac{\beta^u LK\theta}{1-\theta}\,\delta.
\end{split}
\]
Similarly,
\[
U_{n+1} < \dfrac{1}{1+\mu_n}U_n < \left(\prod_{m=0}^{n-1}\dfrac{1}{1+\mu_m}\right)U_0.
\]
Since
$$\prod_{m=0}^{n-1}\dfrac{1}{1+\mu_m} \to 0 \quad \text{as} \quad {n \to +\infty},$$
given $\delta>0$, we have $|U_{n+1}|<M\delta$, where $M=\beta^uLK\theta/(1-\theta)$, for sufficiently large $n$. We conclude that $|U_n|\to 0$ as $n \to +\infty$ and thus
\begin{equation}\label{eq:Sn-to-Sn*}
S_n\to s_n^* \quad \text{as} \quad n \to +\infty.
\end{equation}

By the third equation in~\eqref{eq:principal-disc}, we have, for sufficiently large $n$,
\[
\begin{split}
P_{n+1}-P_n
& =(r_n-b_nP_{n+1})P_n+\theta_n\eta_ng_0(S_n,I_n)P_nI_{n+1}\\
& \le (r_n-b_nP_{n+1})P_n+\theta_n\eta_ng_0(S_n+I_n,0)P_nI_{n+1}\\
& \le (r_n-b_nP_{n+1})P_n+\theta_n\eta_ng_0(s_n^*+2\delta,0)P_n\delta
\end{split}
\]
and thus
$$(r_n-b_nP_{n+1})P_n\le P_{n+1}-P_n\le (r_n+\theta^u\eta^ug_0(s_n^*+2\delta,0)\delta-b_nP_{n+1})P_n.$$
We conclude that
$$\dfrac{r_nP_n+P_n}{1+b_nP_n} \le P_{n+1}\le \dfrac{\left(r_n+\theta^u\eta^ug_0(s_n^*+2\delta,0)\delta\right)P_n+P_n}{1+b_nP_n}.$$
By~\ref{cond-3-aux-y}) in Lemma~\ref{lemma:aux-P}, we have $|P_n-y_n^*| \to 0$ as $n \to +\infty$.
The result follows since $(S_n,I_n,P_n)\to(s^*_n,0,y^*_n)$ as $n \to +\infty$.
\end{proof}

\begin{theorem}\label{teo:Main-ext}
If there is a constant $\lambda \in \N$ such that $\cR^\ell(\lambda)>1$ then the infectives $(I_n)$ are strong persistent in system~\eqref{eq:principal-disc}.
\end{theorem}

\begin{proof}
Assume that there is a constant $\lambda>0$ such that $\cR^\ell(\lambda)>1$. Then, there is a function $\psi$ such that, for all $\delta>0$ sufficiently small we have
\begin{equation}\label{eq:condi-per}
\prod_{k=n}^{n+\lambda} \dfrac{1+\beta_k(x^*_{k+1}-\delta_0)}{1+c_k+\eta_kg(s^*_k-\delta_0,0,z^*_k+\delta_0)}>1+\psi(\delta),
\end{equation}
with $\psi(\delta)>0$ for all $\delta>0$ and $\psi(\delta)\to 0$ as $\delta \to 0$.
Let $N_1 \in \N$ and $(S_n,I_n,P_n)$ be a solution of~\eqref{eq:principal-disc} with $I_n>0$ for all $n \ge N_1$. We will use a contradiction argument to prove that there is $\eps_1>0$ such that
    \begin{equation}\label{eq:weak-persist}
    \limsup_{n \to +\infty} I_n >\eps_1.
    \end{equation}
We may assume that $\eps_1>0$ is sufficiently small so that {H8)} and {H9)} hold for $\eps_1$.
Assuming that~\eqref{eq:weak-persist} does not hold, there is $N_2 \ge N_1$ such that $I_n<\eps_1$ for all $n \ge N_2$.
By the first and third equation in~\eqref{eq:principal-disc}, we conclude that
\[
\begin{cases}
S_{n+1}-S_n\le\Lambda_n-\mu_nS_{n+1}-a_n{f(S_{n+1},\eps_1,P_n)P_n}\\
P_{n+1}-P_n\le(r_n-b_nP_{n+1})P_n+\gamma_na_n{f(S_{n+1},0,P_n)P_n}+\theta_n\eta_ng(S_n,0,P_n)P_n\eps_1
\end{cases},
\]
for all $n \ge N_2$. Considering system~\eqref{eq:aux2b-disc} with $\eps=\eps_1$, we have $S_n \le x_{2,\eps_1,n}$ and $P_n \le z_{2,\eps_1,n}$ for sufficiently large $n$. By {H9)}  we also have, for sufficiently large $n$,
\[
x_{2,\eps_1,n} \le x_{2,\eps_1,n}^*+\eps_1 \quad \text{ and } \quad z_{2,\eps_1,n} \le z_{2,\eps_1,n}^*+\eps_1
\]
and by the continuity properties in H\ref{Cond-4}) and H\ref{Cond-5}), we have
\[
x_{2,\eps_1,n} \le x_{2,\eps_1,n}^*+\eps_1 \le x_{2,n}^*+\chi_1(\eps_1) \quad \text{ and } \quad z_{2,\eps_1,n} \le z_{2,\eps_1,n}^*+\eps_1 \le
z_{2,\eps_1,n}^*+ \chi_2(\eps_1),
\]
with $\chi_1(\eps_1),\chi_2(\eps_1) \to 0$ as $\eps_1\to 0$. Thus, in particular, for sufficiently large $n$,
\begin{equation}\label{eq:Pn-persist}
    S_n \le x^*_{2,\eps,n} \le x_{2,n}^*+\chi_1(\eps_1) \quad \text{ and } \quad P_n \le z^*_{2,\eps,n} \le z_{2,\eps_1,n}^*+ \chi_2(\eps_1),
\end{equation}

Again by the first and third equation in~\eqref{eq:principal-disc}, we conclude that
\[
\begin{cases}
S_{n+1}-S_n\ge\Lambda_n-\mu_nS_{n+1}-a_n{f(S_{n+1},0,P_n)P_n}-\beta_n S_{n+1}\eps_1\\
P_{n+1}-P_n\ge(r_n-b_nP_{n+1})P_n+\gamma_na_n{f(S_{n+1},\eps_1,P_n)P_n}
\end{cases},
\]
for all $n \ge N_2$.

Consider system~\eqref{eq:aux2a-disc} with $\eps=\eps_1$. We have $S_n \ge x_{1,\eps_1,n}$ and $P_n \ge z_{1,\eps_1,n}$ for sufficiently large $n$. By {H8)}  we also have, for sufficiently large $n$,
\[
x_{1,\eps_1,n} \ge x_{1,\eps_1,n}^*-\eps_1 \quad \text{ and } \quad z_{1,\eps_1,n} \ge z_{1,\eps_1,n}^*-\eps_1.
\]
and by the continuity properties in H\ref{Cond-4}) and H\ref{Cond-5}), we have
\[
x_{1,\eps_1,n} \ge x_{1,\eps_1,n}^*-\eps_1 \ge x_{1,n}^*-\phi_1(\eps_1) \quad \text{ and } \quad z_{1,\eps_1,n} \ge z_{1,\eps_1,n}^*-\eps_1 \ge
z_{1,\eps_1,n}^*- \phi_2(\eps_1),
\]
with $\phi_1(\eps_1),\phi_2(\eps_1) \to 0$ as $\eps_1\to 0$.
Thus, in particular, for sufficiently large $n$,
\begin{equation}\label{eq:Sn-persist}
    S_n \ge x_{1,\eps,n} \ge x_{1,n}^*-\phi_1(\eps_1) \quad \text{ and } \quad P_n \ge z_{1,\eps,n} \ge z_{1,n}^*-\phi_2(\eps_1).
\end{equation}

From the second equation in~\eqref{eq:principal-disc},~\eqref{eq:Sn-persist},~\eqref{eq:Pn-persist} and~\eqref{eq:condi-per}, we conclude that
\begin{equation}\label{eq-gerar-contrad}
\begin{split}
I_{n+1}
& =\dfrac{\beta_nS_{n+1}I_n+I_n}{1+\eta_ng(S_n,I_n,P_n)+c_n}\\
&\ge \dfrac{\beta_n( x_{1,\eps_1,n}^*-\eps_1)+1}{1+\eta_ng( x_{1,\eps_1,n}^*-\eps_1,0,z_{2,\eps_1,n}^*+\eps_1)+c_n}\, I_n\\
& >(1+\psi(\eps_1))\,I_{n-\lambda-1}\\
&>\cdots>(1+\psi(\eps_1))^{\lfloor n/(\lambda+1)\rfloor} \, I_{n-\lfloor n/(\lambda+1)\rfloor(\lambda+1)},
\end{split}
\end{equation}
for all $n\ge N_3$ with $N_3 \ge N_2$. Therefore, by~\eqref{eq:condi-per} and~\eqref{eq-gerar-contrad}, we conclude that $I_n \to +\infty$. A contradiction to Lemma~\ref{lemma:region}. We have~\eqref{eq:weak-persist} and the infectives in system~\eqref{eq:principal-disc} are weak persistent.

Using again a contradiction argument, we will prove that we have strong persistence of the infectives. We may assume, with no loss of generality, that
there are $\delta,\delta_0>0$ such that
\begin{equation}\label{eq:condi-per2}
\prod_{k=n}^{n+\lambda} \dfrac{1+\beta_k(x^*_{k+1}-\delta_0)}{1+c_k+\eta_kg(x^*_k-\delta_0,0,z^*_k+\delta_0)}>1+\delta,
\end{equation}
for all sufficiently large $n \in \N$. For each $z_0=(S_0,I_0,P_0)$, denote by $((S_{n,z_0},I_{n,z_0},P_{n,z_0}))$ the solution of~\eqref{eq:principal-disc}  with {$(S_{0,z_0},I_{0,z_0},P_{0,z_0})=(S_0,I_0,P_0)$}.

Proceeding by contradiction, if the system is not strong persistent, then there is a sequence of initial values $z_{0,k}=(S_{0,k},I_{0,k},P_{0,k})$, $k \in\N$, such that
    \begin{equation}\label{eq:contrad-2}
    \liminf_{n \to +\infty} I_{n,z_{0,k}} < \frac{\eps_0}{k^2}.
    \end{equation}
From~\eqref{eq:weak-persist} and~\eqref{eq:condi-per2}, for each $k \in\N$ there are sequences $(s_{m,k})$ and $(t_{m,k})$ such that
\begin{equation}\label{eq:prop-seq-1}
0<s_{1,k}<t_{1,k}<s_{2,k}<t_{2,k}<\cdots<s_{m,k}<t_{m,k}<\cdots,
\end{equation}
\begin{equation}\label{eq:prop-seq-2}
s_{m,k} \to +\infty \ \ \text{as} \ \ m \to +\infty,
\end{equation}
\begin{equation}\label{eq:prop-seq-3}
I_{s_{m,k},z_{0,k}}>\dfrac{\eps_0}{k}, \ \  I_{t_{m,k},z_{0,k}}<\dfrac{\eps_0}{k^2},
\end{equation}
and
\begin{equation}\label{eq:prop-seq-4}
\frac{\eps_0}{k^2}\le I_{n,z_{0,k}}\le \frac{\eps_0}{k}, \ \  \text{for all} \ \ n \in [s_{m,k},t_{m,k}-1]\cap\N.
\end{equation}

For any $n \in [s_{m,k},t_{m,k}-1]\cap\N$ sufficiently large, we have, using~\eqref{eq:bound},
\[
\begin{split}
I_{n+1,z_{0,k}}
&=\dfrac{1+\beta_nS_{n+1,z_{0,k}}}{1+c_n+\eta_n g(S_{n,z_{0,k}},I_{n,z_{0,k}},P_{n,z_{0,k}})}I_{n,z_{0,k}}\\
& > \dfrac{1}{1+c_n+\eta_n g(S_{n,z_{0,k}},0,P_{n,z_{0,k}})}I_{n,z_{0,k}}\\
& \ge\dfrac{1}{1+a}I_{n,z_{0,k}},
\end{split}
\]
where $a=c^u+\eta^u g(0,0,L+\delta)>0$. Therefore, by~\eqref{eq:prop-seq-3}, we obtain
$$
\frac{\eps_0}{k^2} > I_{t_{m,k},z_{0,k}} \ge \left(\dfrac{1}{1+a}\right)^{t_{m,k}-s_{m,k}} I_{s_{m,k},z_{0,k}}
> \left(\dfrac{1}{1+a}\right)^{t_{m,k}-s_{m,k}} \frac{\eps_0}{k},
$$
and therefore we get
$$t_{m,k}-s_{m,k} > \dfrac{\ln k}{\ln(1+a)}\to +\infty \ \ \text{as} \ \ k \to +\infty. $$
We conclude that we can choose $k_1 \in \N$ such that
$$t_{m,k}-s_{m,k} > n_1+\lambda+1,$$
for all $k \ge k_1$.

Now, for all $k \ge k_1$ and $n \in [s_{m,k}+1,t_{m,k}]\cap\N$, we have
\[
\begin{cases}
S_{n+1,z_{0,k}}-S_{n,z_{0,k}}& \le\Lambda_n-\mu_nS_{n+1,z_{0,k}}-a_n{f(S_{n+1,z_{0,k}},P_{n,z_{0,k}})P_{n,z_{0,k}}}\\
P_{n+1,z_{0,k}}-P_{n,z_{0,k}}& \le(r_n-b_nP_{n+1,z_{0,k}})P_{n,z_{0,k}}\\
& \phantom{\le} +\gamma_na_n{f(S_{n+1,z_{0,k}},P_{n,z_{0,k}})P_{n,z_{0,k}}}+\theta_n\eta_ng(S_{n,z_{0,k}},I_{n,z_{0,k}},P_{n,z_{0,k}})\eps_1
\end{cases},
\]
Let $(\bar x_n, \bar z_n)$ be a solution of~\eqref{eq:aux2-disc} with initial condition $\bar x_{s_{m,k}+1}=S_{s_{m,k}+1}$ and $\bar z_{s_{m,k}+1}=P_{s_{m,k}+1}$. By {H8)}, for sufficiently large $k \in \N$ we have
\[
|S_{n,z_0,k}-x^*_n|\le |S_{n,z_0,k}-\bar x_n|+|\bar x_n-x^*_n| < \eps_0/2+\eps_0/2=\eps_0
\]
for all $n \in [s_{m,k}+1,t_{m,k}]\cap\N$. In particular
\begin{equation}\label{eq:final-eq-a}
S_{n,z_0,k} \ge x^*_n - \eps_0,
\end{equation}
for all $n \in [s_{m,k}+1,t_{m,k}]\cap\N$.
In a similar way, using {H9)}, we conclude that, for sufficiently large $k \in \N$ we have
\begin{equation}\label{eq:final-eq-a}
P_{n,z_0,k} \le y^*_n + \eps_0,
\end{equation}
for all $n \in [s_{m,k}+1,t_{m,k}]\cap\N$.

Finally, we have
\begin{equation}\label{eq:ultima--1}
\begin{split}
I_{n+1,z_{0,k}}
& =\dfrac{1+\beta_nS_{n+1,z_{0,k}}}{1+c_n+\eta_n g(S_{n,z_{0,k}},I_{n,z_{0,k}},P_{n,z_{0,k}})}I_{n,z_{0,k}}\\
& \ge \dfrac{1+\beta_n(x^*_n-\eps_0)}{1+c_n+\eta_n g(x^*_n - \eps_0,0,y^*_n+\eps_0)}I_{n,z_{0,k}}
\end{split}
\end{equation}
for all $n \in [s_{m,k}+n_1+1,t_{m,k}]\cap\N$ and $k \ge n_4$. By~\eqref{eq:contrad-2} and~\eqref{eq:ultima--1} we get
\[
\frac{\eps_0}{k^2}
> I_{t_{m,k},z_{0,k}}
\ge I_{t_{m,k}-\lambda,z_{0,k}} \prod_{n=t_{m,k}-\lambda}^{t_{m,k}}
\dfrac{1+\beta_n(x^*_n-\eps_0)}{1+c_n+\eta_n g(x^*_n - \eps_0,0,y^*_n+\eps_0)}I_{n,z_{0,k}} > \frac{\eps_0}{k^2},
\]
a contradiction. Thus we conclude that the infectives are strong persistent and the result follow.
\end{proof}

\section{Examples}\label{sec3}

\subsection{A model with no predation of uninfected preys}

Letting $a \equiv 0$ and $g(x,y,z)=z$ in~\eqref{eq:principal-disc}, we obtain the model below that corresponds to the discrete counterpart of the model in~\cite{Niu-Zhang-Teng-AMM-2011}.
\begin{equation}\label{eq:principal-disc-Zhang-Niu-Teng}
\begin{cases}
S_{n+1}-S_n=\Lambda_n-\mu_nS_{n+1}-\beta_nS_{n+1}I_n\\
I_{n+1}-I_n=\beta_nS_{n+1}I_n-\eta_nP_nI_{n+1}-c_nI_{n+1}\\
P_{n+1}-P_n=(r_n-b_nP_{n+1})P_n+\theta_n\eta_nP_nI_{n+1}
\end{cases}.
\end{equation}

For model~\eqref{eq:principal-disc-Zhang-Niu-Teng} we assume conditions H\ref{Cond-1}), H\ref{Cond-1a}) and H\ref{Cond-2}). Notice that H\ref{cond-f}) is trivial,  H\ref{Cond-aditional}) and H\ref{Cond-aditional2}) follow from the discussion on~\eqref{eq:principal-disc-explicit} with $a_n=0$, H\ref{Cond-7a}) follows from Lemma~\ref{lemma:region} and H\ref{Cond-4}) and  H\ref{Cond-5}) follow from Lemma \ref{lemma:aux-S} and Lemma~\ref{lemma:aux-P}, respectively.

For each solution $(s_n^*)$ of~\eqref{eq:aux-syst-S} with $s_0>0$, each solution $(y_n^*)$ of~\eqref{eq:aux-syst-P} with $y_0>0$ and each $\lambda \in \N$, in this context of no predation (of uninfected preys) we set
\[
\cR_{NP}^\ell(\lambda) = \liminf_{n \to \, +\infty} \prod_{k=n}^{n+\lambda} \dfrac{1+\beta_ks^*_{k+1}}{1+c_k+\eta_ky^*_k}
\]
and
\[
\cR_{NP}^u(\lambda) = \limsup_{n \to \, +\infty} \prod_{k=n}^{n+\lambda} \dfrac{1+\beta_ks^*_{k+1}}{1+c_k+\eta_ky^*_k}.
\]

The next theorems correspond to discrete counterparts of the results in~\cite{Niu-Zhang-Teng-AMM-2011}.

\begin{theorem}\label{teo:Main-ext-Zhang-Niu-Teng}
If there is $\lambda \in \N$ such that $\cR_{NP}^u(\lambda)<1$ then the infectives $(I_n)$ go to extinction in system~\eqref{eq:principal-disc-Zhang-Niu-Teng} and any disease-free solution $((s^*_n,0,y^*_n))$ of~\eqref{eq:principal-disc-Zhang-Niu-Teng}, where $(s^*_n)$ is a solution of~\eqref{eq:aux-syst-S} and $(y^*_n)$ is a solution of~\eqref{eq:aux-syst-P}, is globally asymptotically attractive.
\end{theorem}

\begin{theorem}\label{teo:Main-ext-Zhang-Niu-Teng}
If there is $\lambda \in \N$ such that $\cR_{NP}^\ell(\lambda)>1$ then the infectives $(I_n)$ are strongly persistent in system~\eqref{eq:principal-disc-Zhang-Niu-Teng}.
\end{theorem}

To do some simulation, we consider the particular solutions $s^*_n=\Lambda/\mu$, $y^*_n=r/b$ and the following particular set of parameters in system~\eqref{eq:principal-disc-Zhang-Niu-Teng}: $\Lambda_n=0.3$, $\mu_n=0.1$, $\beta_n=\beta_0(1+0.7\cos(\pi n/5))$, $\eta_n=0.3(1+0.7\cos(\pi n/5))$, $c_n=0.18$, $r_n=0.3$, $b_n=0.2$, $\theta_n=0.9$. This example is based on a continuous-time example in~\cite{Jesus-Silva-Vilarinho-preprint1-2020}.

When $\beta_0=0.17$ we obtain $\cR_{NP}^u(\lambda)\approx 0.89<1$ and we conclude that we have the extinction (figure~\ref{fig-no-predation-on-uninfected-PER-1}). When $\beta_0=0.29$ we obtain $\cR_{NP}^\ell(\lambda)\approx 1.24>1$ and we conclude that the infectives are strongly persistent (figure~\ref{fig-no-predation-on-uninfected-DF-1}).

In extinction and uniform strong persistence scenario we considered, respectively, the following initial conditions: $(S_0,I_0,P_0)=(0.8,0.6,0.1)$, $(S_0,I_0,P_0)=(1.7,0.2,0.3)$ and $(S_0,I_0,P_0)=(2.3,0.4,0.7)$; $(S_0,I_0,P_0)=(1.5,0.1,0.2)$, $(S_0,I_0,P_0)=(0.7,0.2,0.4)$ and $(S_0,I_0,P_0)=(0.3,0.15,0.9)$.
\begin{figure}
  \begin{minipage}[b]{.32\linewidth}
    \includegraphics[width=\linewidth]{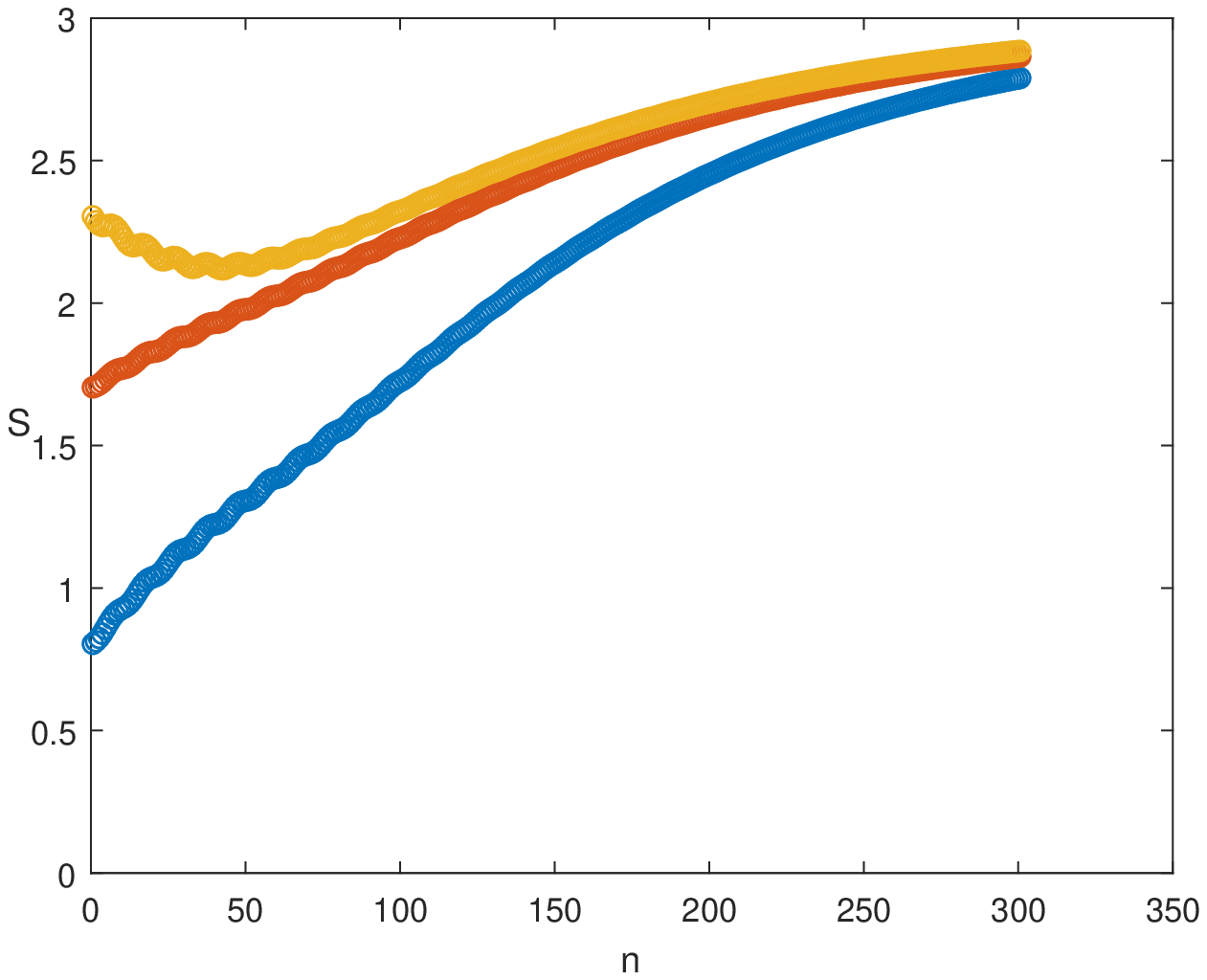}
  \end{minipage}
  \begin{minipage}[b]{.32\linewidth}
        \includegraphics[width=\linewidth]{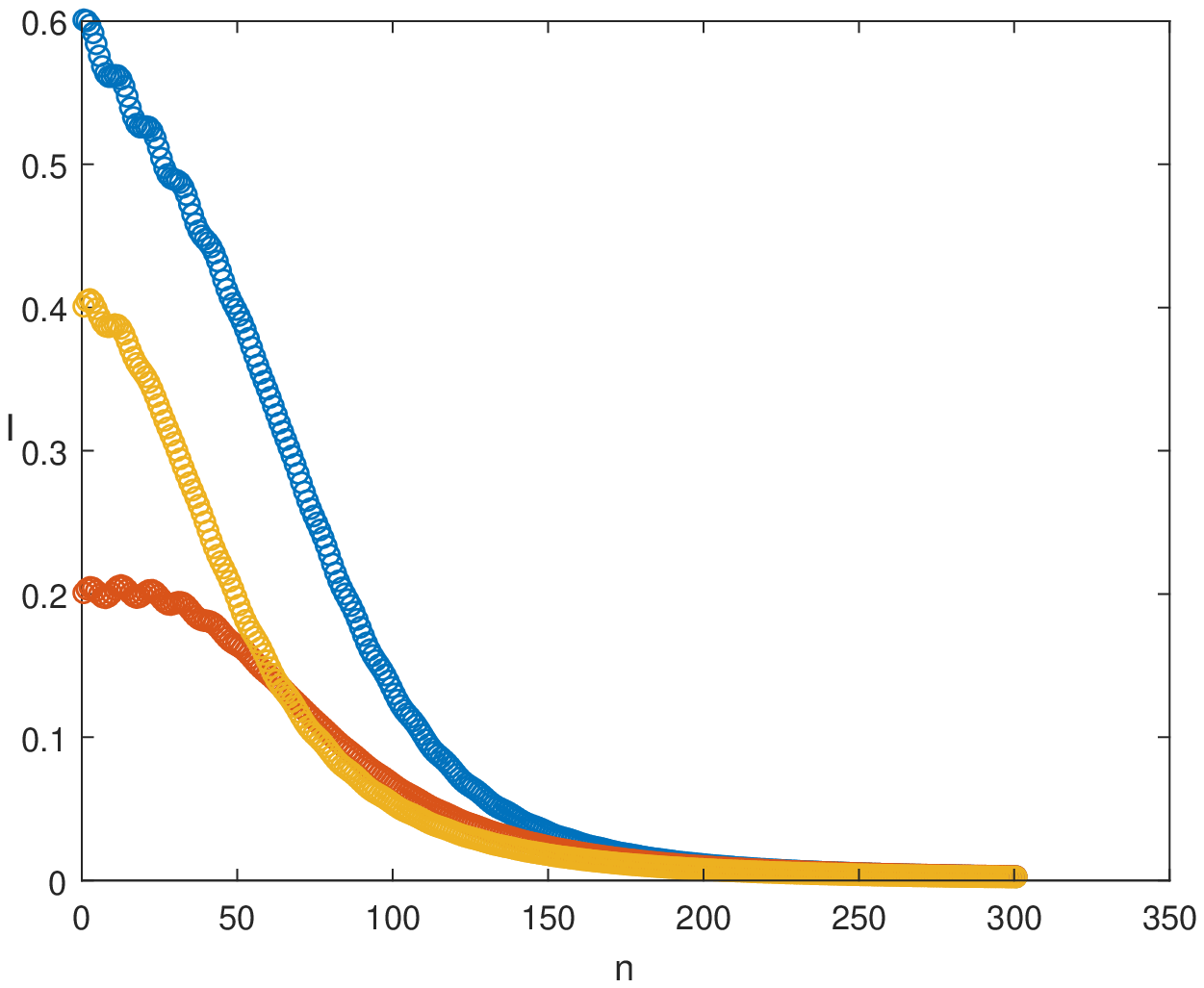}
  \end{minipage}
  \begin{minipage}[b]{.32\linewidth}
        \includegraphics[width=\linewidth]{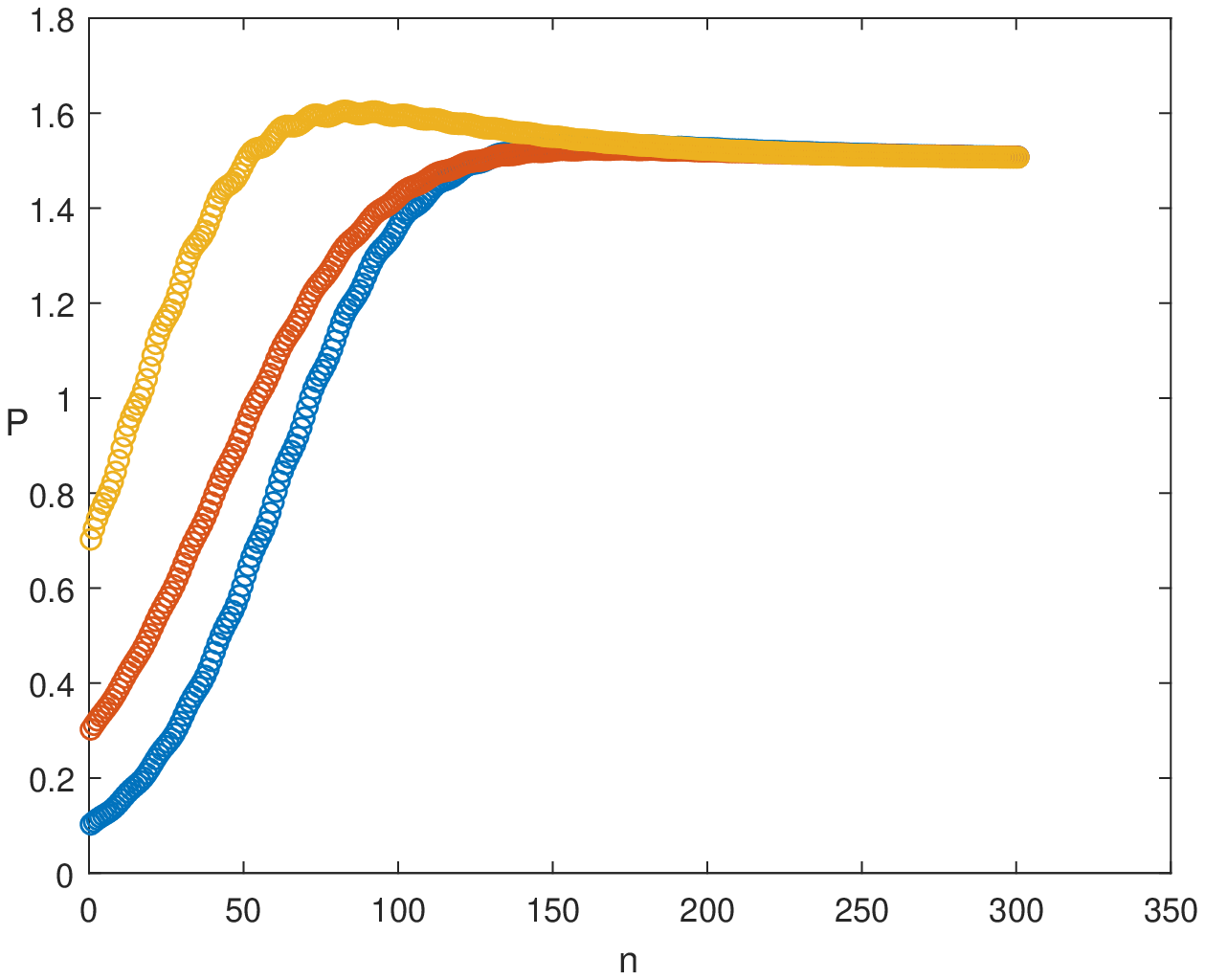}
  \end{minipage}
    \caption{Extinction, when $\beta_0=0.17$.}
      \label{fig-no-predation-on-uninfected-PER-1}
\end{figure}
\begin{figure}
  \begin{minipage}[b]{.32\linewidth}
    \includegraphics[width=\linewidth]{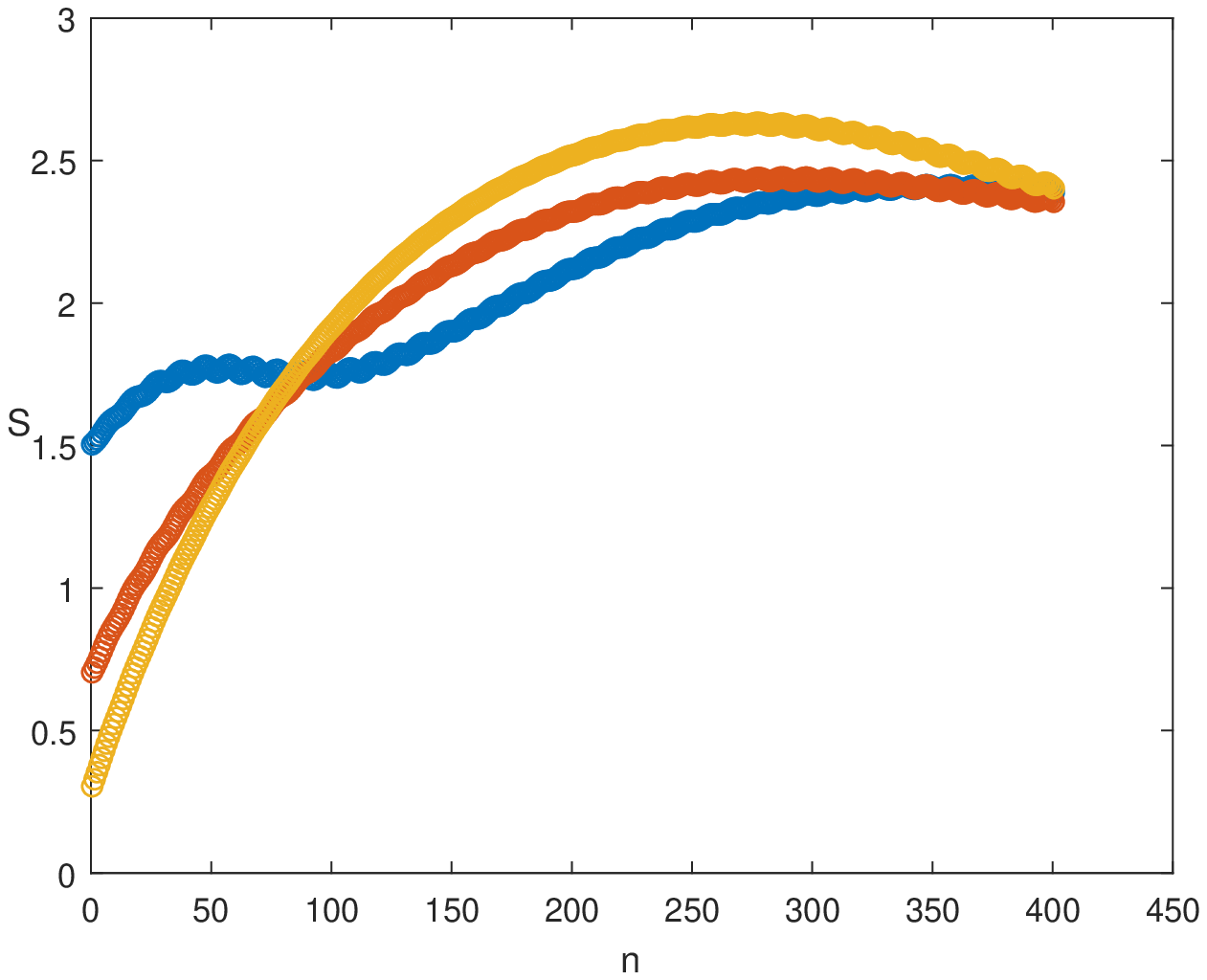}
  \end{minipage}
  \begin{minipage}[b]{.32\linewidth}
        \includegraphics[width=\linewidth]{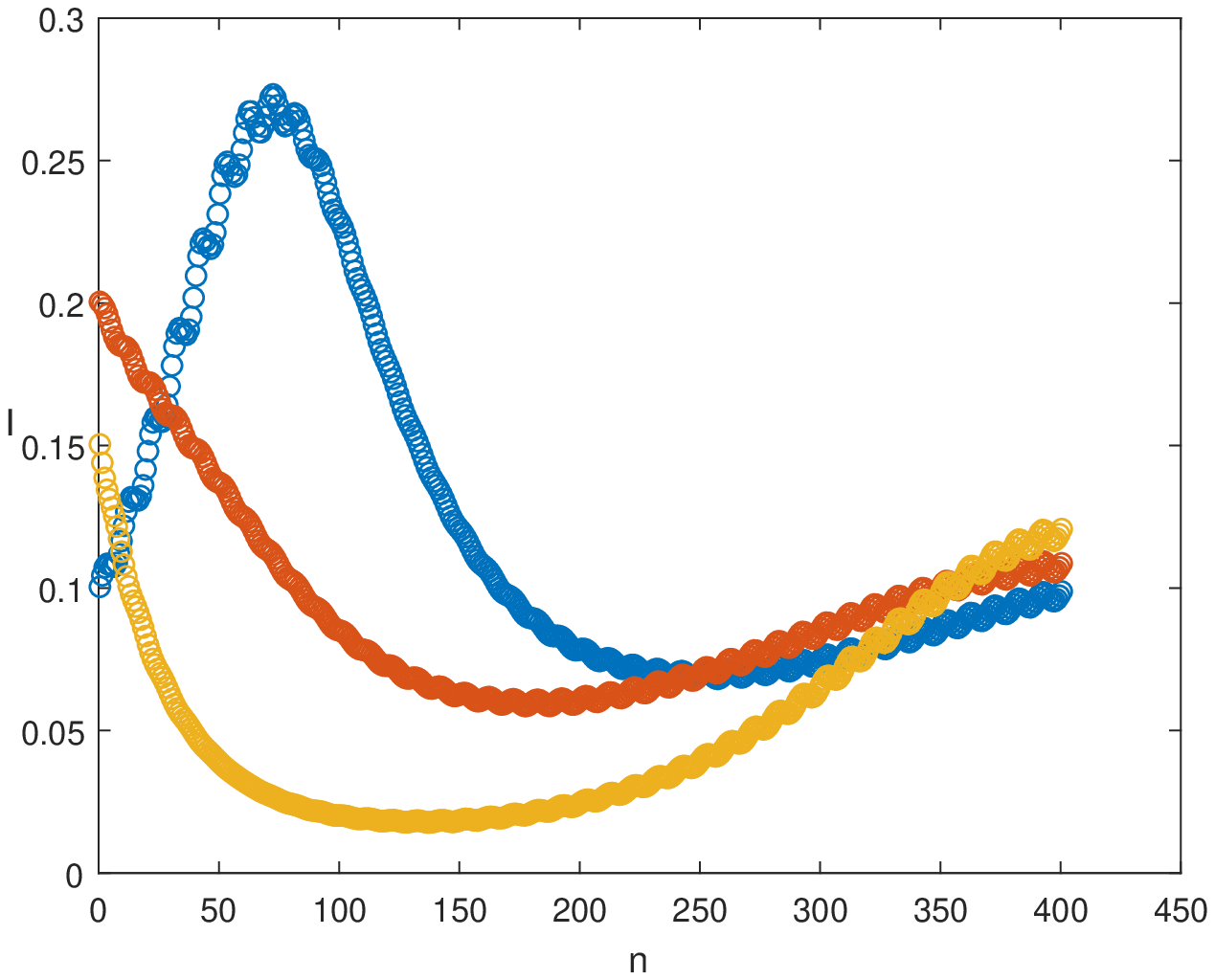}
  \end{minipage}
  \begin{minipage}[b]{.32\linewidth}
        \includegraphics[width=\linewidth]{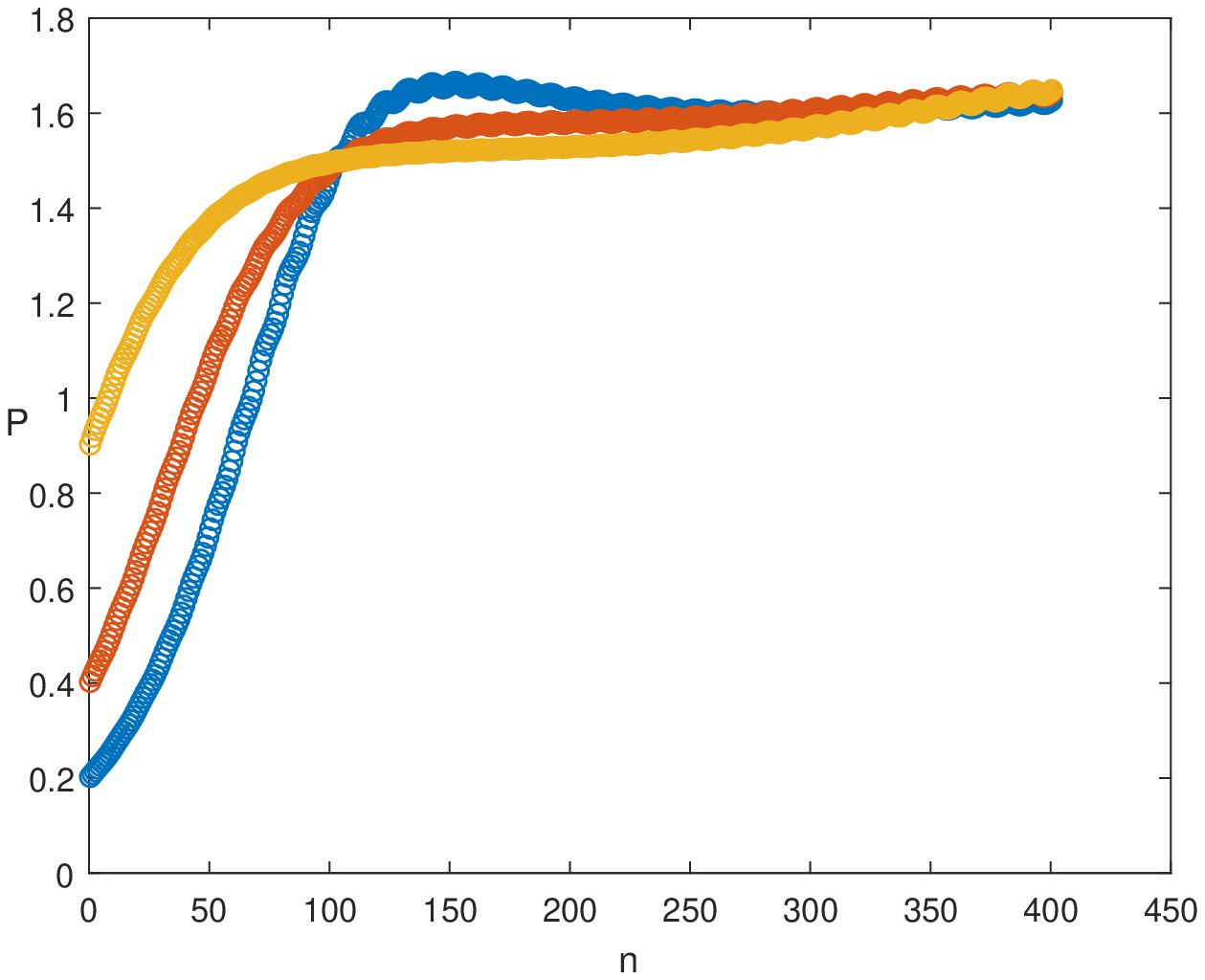}
  \end{minipage}
    \caption{Strong persistence, when $\beta_0=0.29$.}
      \label{fig-no-predation-on-uninfected-DF-1}
\end{figure}

\subsection{Periodic model}
Consider the system~\eqref{eq:principal-disc} and assume that there is $\omega \in \N$ such that $\Lambda_{n+\omega}=\Lambda_n$, $\mu_{n+\omega}=\mu_n$, $a_{n+\omega}=a_n$, $\beta_{n+\omega}=\beta_n$, $\eta_{n+\omega}=\eta_n$, $c_{n+\omega}=c_n$, $r_{n+\omega}=r_n$, $b_{n+\omega}=b_n$, $\gamma_{n+\omega}=\gamma_n$ and $\theta_{n+\omega}=\theta_n$, for all $n \in \N$. Conditions H\ref{Cond-1}) to H\ref{cond-f}) and  H\ref{Cond-aditional}) to  H\ref{Cond-4}) are assumed; condition H\ref{Cond-2}) is trivial.

For each solution $(s_n^*)$ of~\eqref{eq:aux-syst-S} with $s_0>0$, each solution $(y_n^*)$ of~\eqref{eq:aux-syst-P} with $y_0>0$ and for each solution $((x^*_n,z^*_n))$ of~\eqref{eq:aux2a-disc} with $\eps=0$ and initial conditions $x_0>0$ and $z_0>0$, and each $\lambda \in \N$, we set

\[
\mathcal R_{PER}^\ell=\prod_{k=1}^{\omega}\frac{1+\beta_k x^*_{k+1}}{1+c_k+\eta_k g(x^*_k,0,z^*_k)}
\]
and
\[
\mathcal R_{PER}^u=\prod_{k=1}^{\omega}\frac{1+\beta_k s^*_{k+1}}{1+c_k+\eta_kg(s^*_k,0,y^*_k)},
\]

\begin{corollary}
If $\cR_{PER}^u<1$ then the infective $(I_n)$ go to extinction in system~\eqref{eq:principal-disc} and any disease-free solution $((s^*_n,0,y^*_n))$ of~\eqref{eq:principal-disc}, where $(s^*_n)$ is a solution of the periodic version of~\eqref{eq:aux-syst-S} and $(y^*_n)$ is a solution of the periodic version of~\eqref{eq:aux-syst-P}, is globally asymptotically attractive.
\end{corollary}

\begin{corollary}
If $\cR_{PER}^\ell~>1$ then the infective $(I_n)$ is strongly persistent in system~\eqref{eq:principal-disc}, where $(x^*_n)$ and $(z^*_n)$ are the components of the solution $((x^*_n, z^*_n))$ in the periodic version of~\eqref{eq:aux2a-disc}. Moreover, there exist a periodic orbit of period $\omega.$
\end{corollary}

To do some simulation, we consider $f(x,y,z)=x$, $g(x,y,z)=z$, the particular solutions $s^*_k=\Lambda/\mu$, $y^*_k=r/b$ and $$(x^*_n,z^*_n)=\left(\frac{-K_1+\sqrt{K_1^2+2\Lambda K_2}}{K_2},-\frac{K_1}{2a}+\frac{1}{2a}\sqrt{K_1^2+2\Lambda K_2}+r/b\right),$$
where $K_1=\mu+\frac{ar}{b}$ and $K_2=\frac{2\gamma a^2}{b};$  we also considered the following particular set of parameters, and the exception of $\beta$ and $\eta$ we assume that they are all constants: $\Lambda_n=0.3$, $\mu_n=0.1$, $a_n=0.4$, $\beta_n=\beta_0(1+0.7\cos(\pi n/5))$, $\eta_n=0.3(1+0.7\cos(\pi n/5))$, $c_n=0.18$, $r_n=0.3$, $b_n=0.2$, $\gamma_n=0.1$ and $\theta=0.9$.

When $\beta_0=0.17,$ we obtain $\mathcal R_{PER}^u\approx 0.44<1$ and we conclude that we have the extinction (figure~\ref{fig-no-periodic-PER-1}). When $\beta_0=2.2,$ we obtain $\mathcal R_{PER}^\ell(\lambda)\approx 3.013>1$ and we conclude that the infectives are uniformly strong persistent (figure~\ref{fig-no-periodic-DF-1}).

In extinction and uniform strong persistence scenario we considered, respectively, the following initial conditions: $(S_0,I_0,P_0)=(0.8,0.6,0.1)$, $(S_0,I_0,P_0)=(1.7,0.2,0.3)$ and $(S_0,I_0,P_0)=(2.3,0.4,0.7)$; $(S_0,I_0,P_0)=(1.5,0.1,0.2)$, $(S_0,I_0,P_0)=(0.7,0.2,0.4)$ and $(S_0,I_0,P_0)=(0.3,0.15,0.9)$.
\begin{figure}
  \begin{minipage}[b]{.32\linewidth}
    \includegraphics[width=\linewidth]{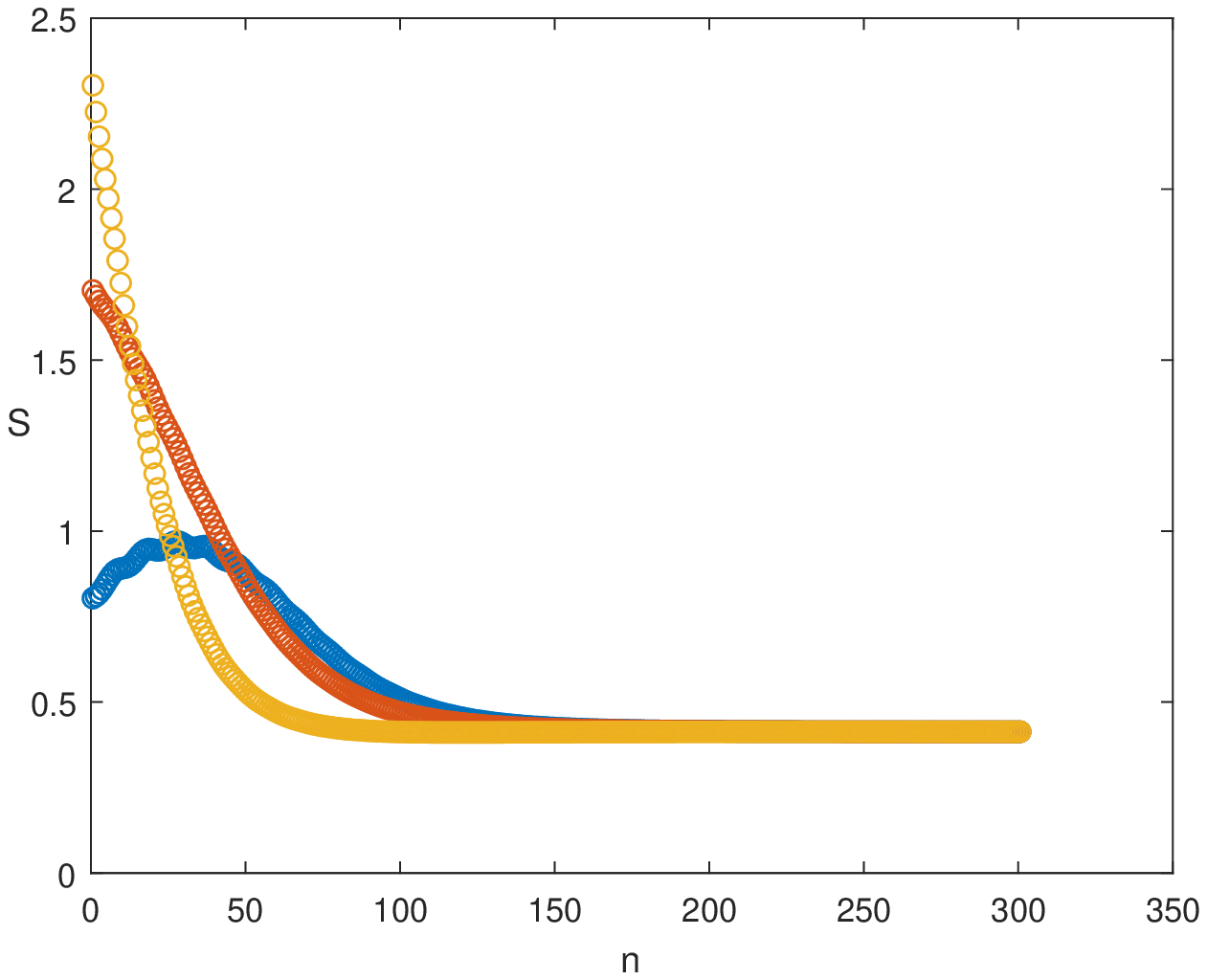}
  \end{minipage}
  \begin{minipage}[b]{.32\linewidth}
        \includegraphics[width=\linewidth]{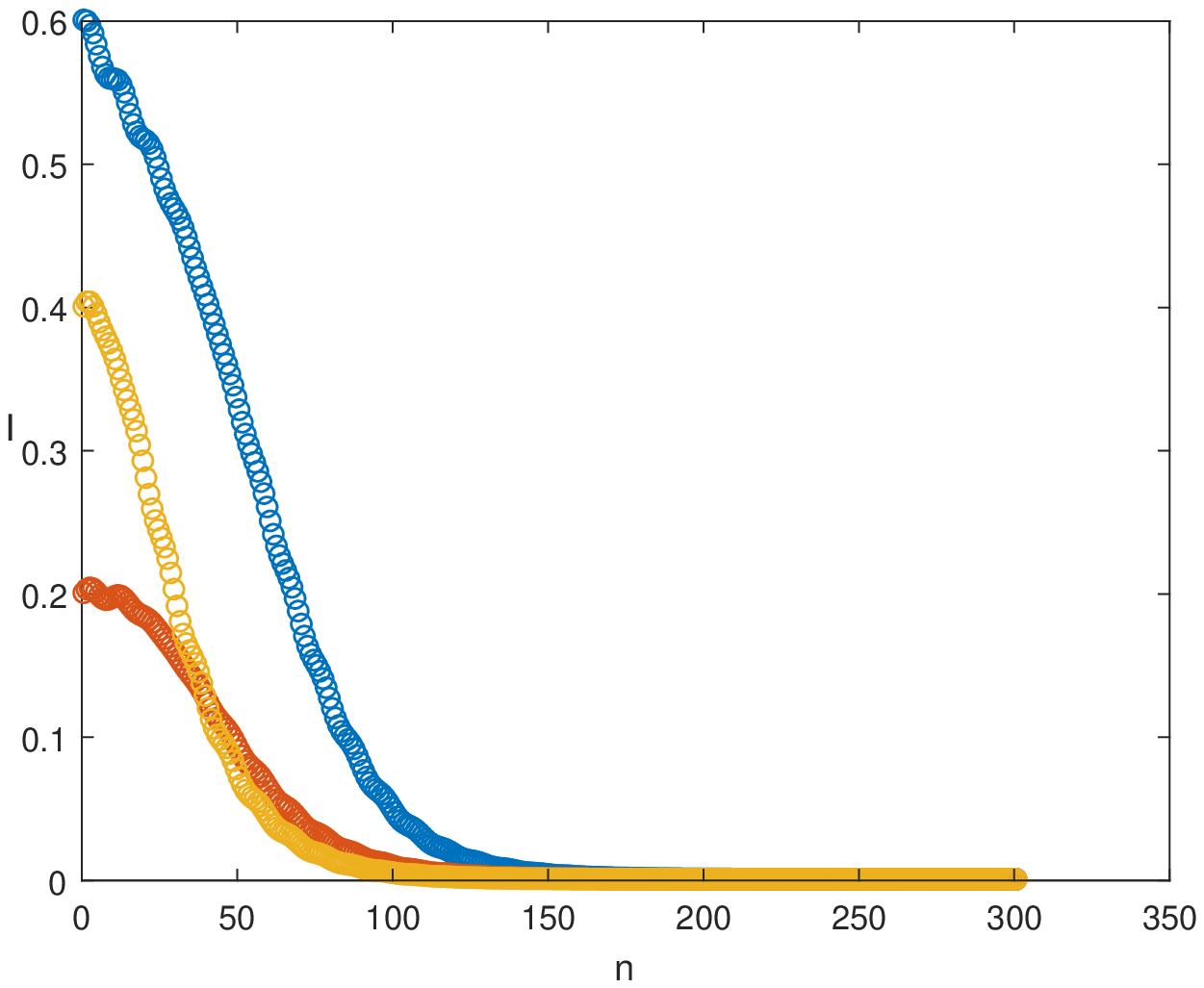}
  \end{minipage}
  \begin{minipage}[b]{.32\linewidth}
        \includegraphics[width=\linewidth]{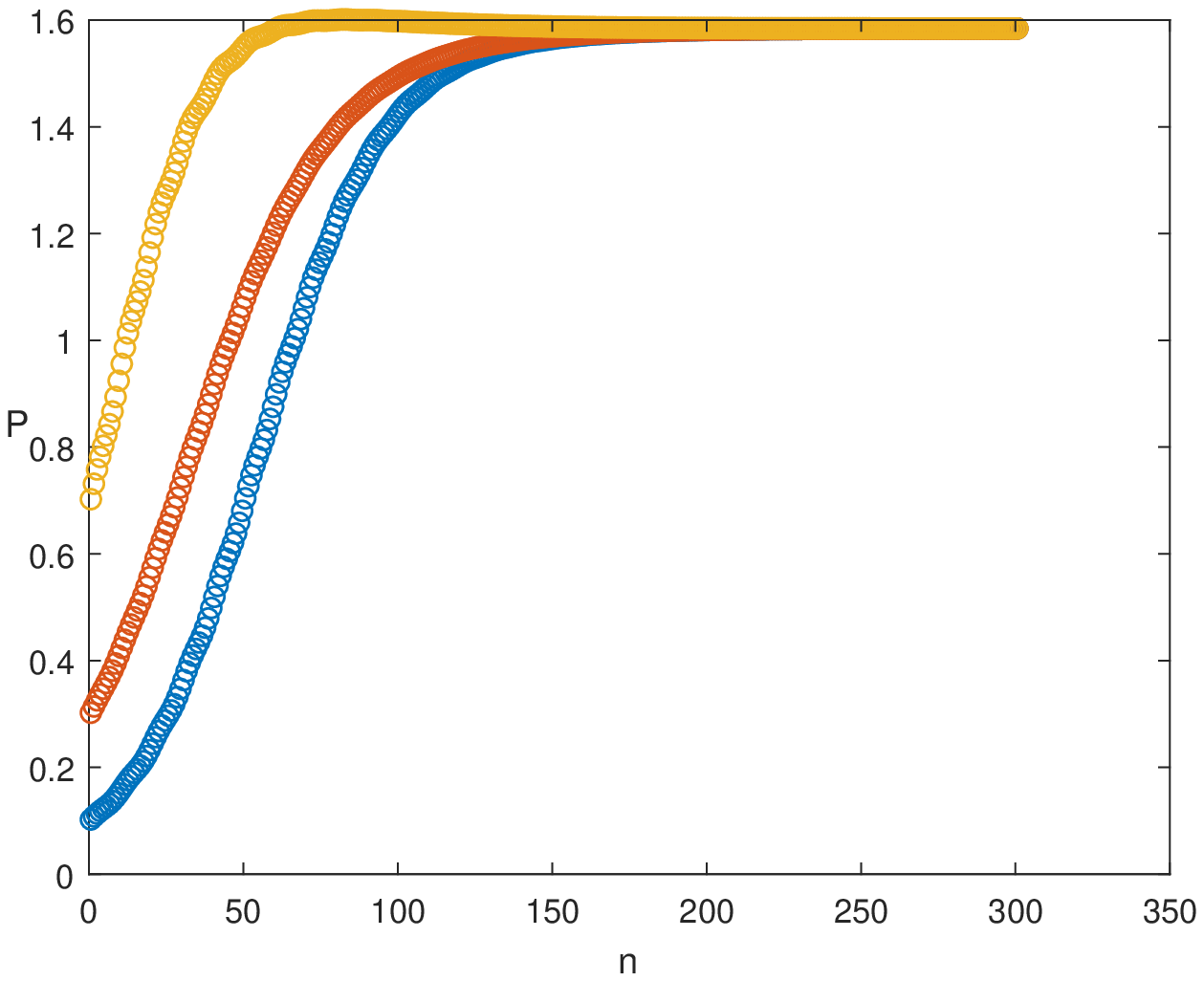}
  \end{minipage}
    \caption{Extinction, when $\beta_0=0.17$.}
      \label{fig-no-periodic-PER-1}
\end{figure}
\begin{figure}
  \begin{minipage}[b]{.32\linewidth}
    \includegraphics[width=\linewidth]{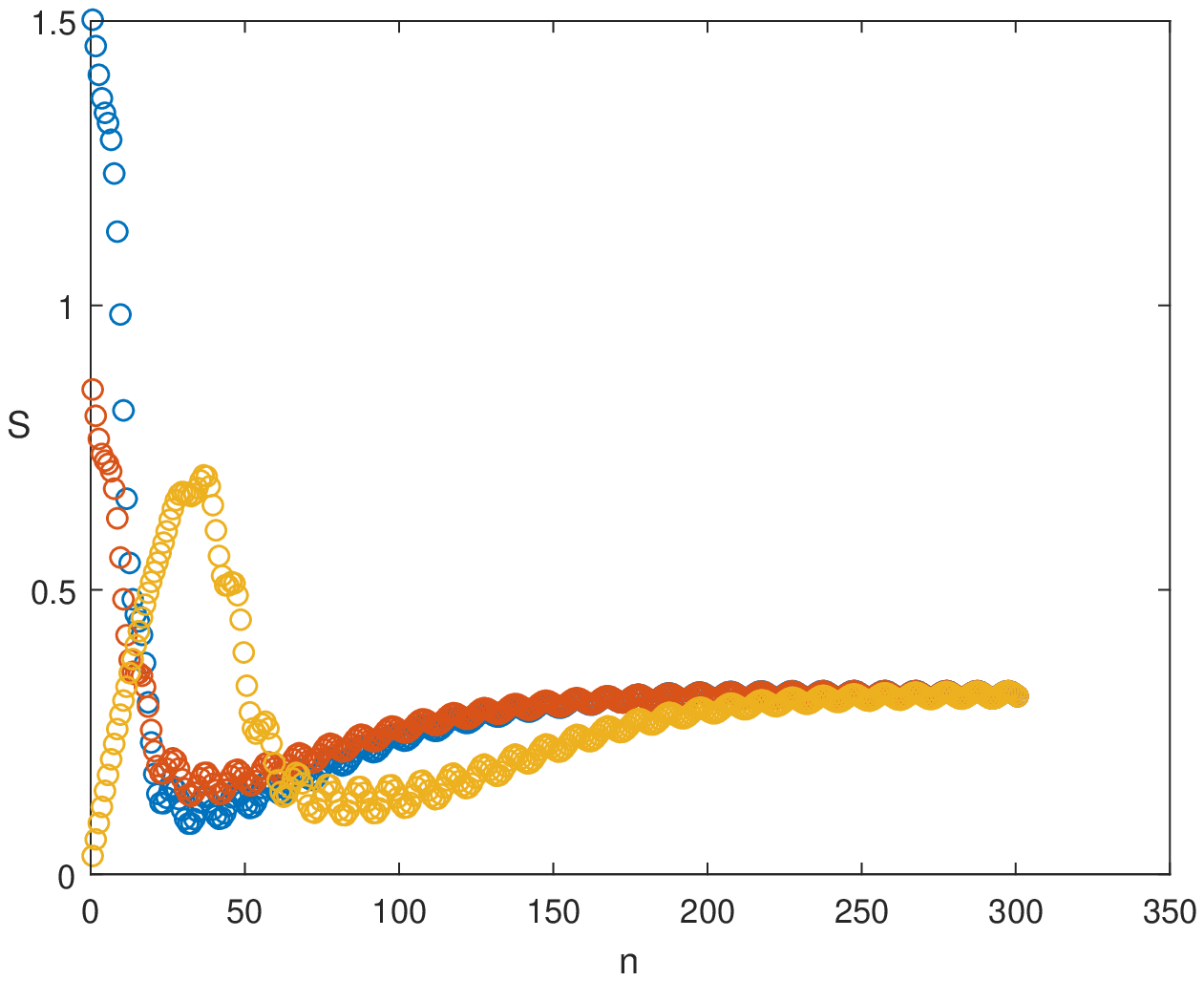}
  \end{minipage}
  \begin{minipage}[b]{.32\linewidth}
        \includegraphics[width=\linewidth]{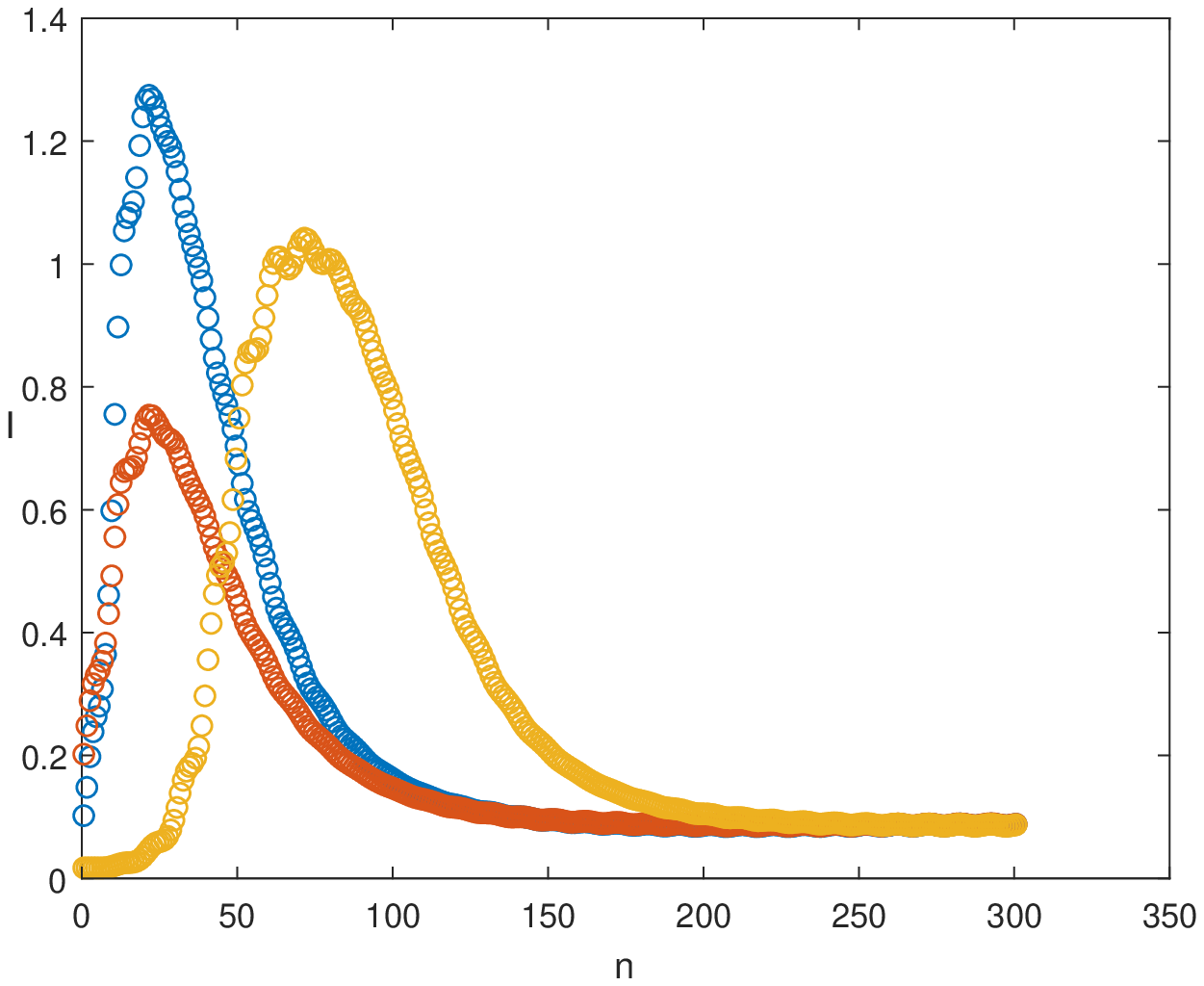}
  \end{minipage}
  \begin{minipage}[b]{.32\linewidth}
        \includegraphics[width=\linewidth]{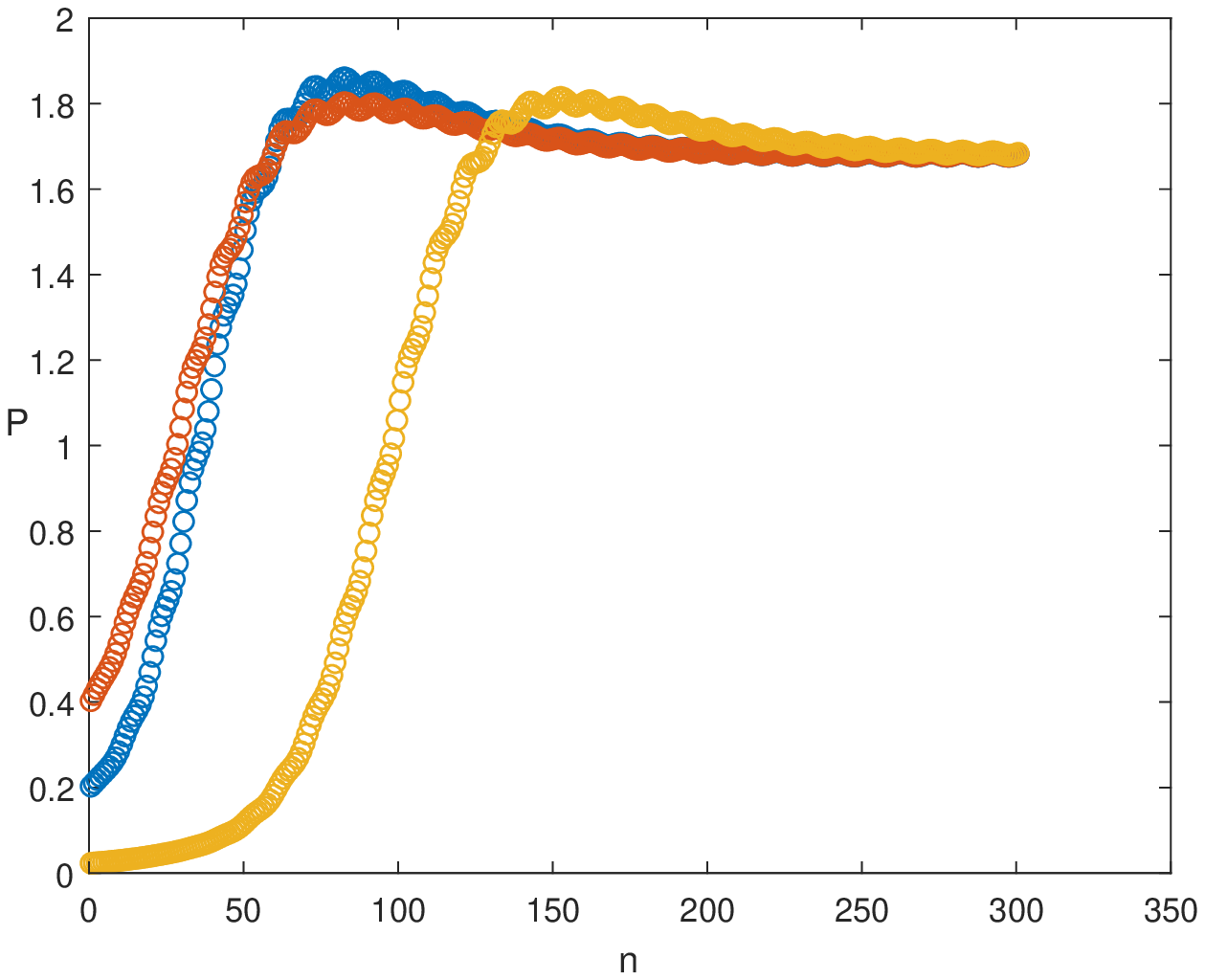}
  \end{minipage}
    \caption{Strong persistence, when $\beta_0=2.2$.}
      \label{fig-no-periodic-DF-1}
\end{figure}

\subsection{Autonomous model}
Consider the system~\eqref{eq:principal-disc}, and assume now that $f(x,y,z)=x$, $g(x,y,z)=z$, $\Lambda_n=\Lambda$, $\mu_n=\mu$, $a_n=a$, $\beta_n=\beta$, $\eta_n=\eta$, $c_n=c$, $r_n=r$, $b_n=b$ and $\gamma_n=\gamma$, $\theta_n=\theta$. Then we obtain following the model:

\begin{equation}\label{eq:autonomous-disc}
\begin{cases}
S_{n+1}-S_n=\Lambda-\mu S_{n+1}-aS_{n+1}P_n-\beta S_{n+1}I_n\\
I_{n+1}-I_n=\beta S_{n+1}I_n-\eta I_{n+1}P_n-cI_{n+1}\\
P_{n+1}-P_n=(r-bP_{n+1})P_n+\gamma aS_{n+1}P_n+\theta\eta I_{n+1}P_n
\end{cases}.
\end{equation}

Conditions H\ref{Cond-1}) to H\ref{Cond-2}) are immediate. Conditions H\ref{Cond-aditional}) and H\ref{Cond-aditional2}) follow from the discussion on~\eqref{eq:principal-disc-explicit}. Condition H\ref{Cond-7a}) follows from Lemma~\ref{lemma:region} and H\ref{Cond-4}) and  H\ref{Cond-5}) follow from Lemma \ref{lemma:aux-S} and Lemma~\ref{lemma:aux-P}, respectively.

For each solution $(s^*_n)$ of~\eqref{eq:aux-syst-S} with $s_0>0$, each solution $(y^*_0)$ of~\eqref{eq:aux-syst-P} with $y_0>0$ and each solution $((x^*_n,z^*_n))$ of~\eqref{eq:aux2a-disc} with $\eps=0$ and initial conditions $x_0>0$ and $z_0>0$, and each $\lambda \in \N$, we set

\[\mathcal R_{A}^\ell=\dfrac{1+\beta\left(\frac{-K_1+\sqrt{K_1^2+2\Lambda K_2}}{K_2}\right) }{1+c+\eta \left(-\frac{K_1}{2a}+\frac{1}{2a}\sqrt{K_1^2+2\Lambda K_2}+r/b\right)}\]
and
\[\mathcal R_{A}^u=\dfrac{1+\beta \left(\Lambda/\mu\right) }{1+c+\eta \left(r/b\right)  },\]
where $K_1=\mu+\frac{ar}{b}$ and $K_2=\frac{2\gamma a^2}{b}.$

\begin{corollary}
If $\mathcal R_{A}^u<1$ then the infective $(I_n)$ in system~\eqref{eq:autonomous-disc} go to extinction.
\end{corollary}

\begin{corollary}
If $\mathcal R_{A}^\ell>1$ then the infective $(I_n)$ in system~\eqref{eq:autonomous-disc} are strongly persistent.
\end{corollary}

To do some simulation, we consider the following particular set of parameters: $\Lambda=0.3$, $\mu=0.1$, $a=0.4$, $\eta=0.3$, $c=0.18$, $r=0.3$, $b=0.2$, $\gamma=0.1$ and $\theta=0.9$.

When $\beta=0.17$ we obtain $\mathcal R_{A}^u\approx 0.93<1$ and we conclude that we have the extinction. When $\beta=2.2$ we obtain $\mathcal R_{A}^\ell(\lambda)\approx 1.14>1$ and we conclude that the infectives are strongly persistent.

In uniform strong persistence and extinction scenario we considered, respectively, the following initial conditions: $(S_0,I_0,P_0)=(0.8,0.6,0.1)$, $(S_0,I_0,P_0)=(1.7,0.2,0.3)$ and $(S_0,I_0,P_0)=(2.3,0.4,0.7)$; $(S_0,I_0,P_0)=(1.5,0.1,0.2)$, $(S_0,I_0,P_0)=(0.7,0.2,0.4)$ and $(S_0,I_0,P_0)=(0.3,0.15,0.9)$.
\begin{figure}
  \begin{minipage}[b]{.32\linewidth}
    \includegraphics[width=\linewidth]{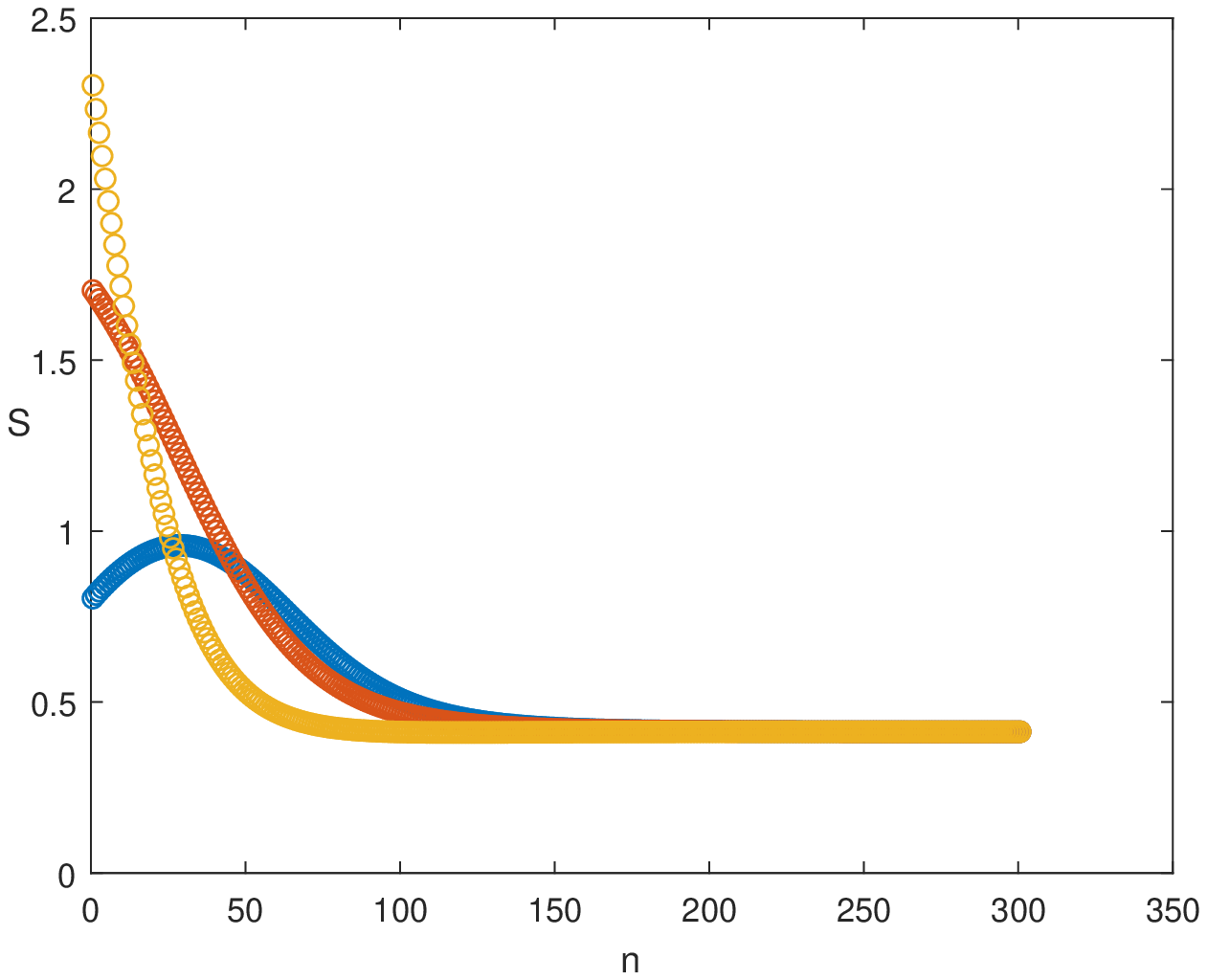}
  \end{minipage}
  \begin{minipage}[b]{.32\linewidth}
        \includegraphics[width=\linewidth]{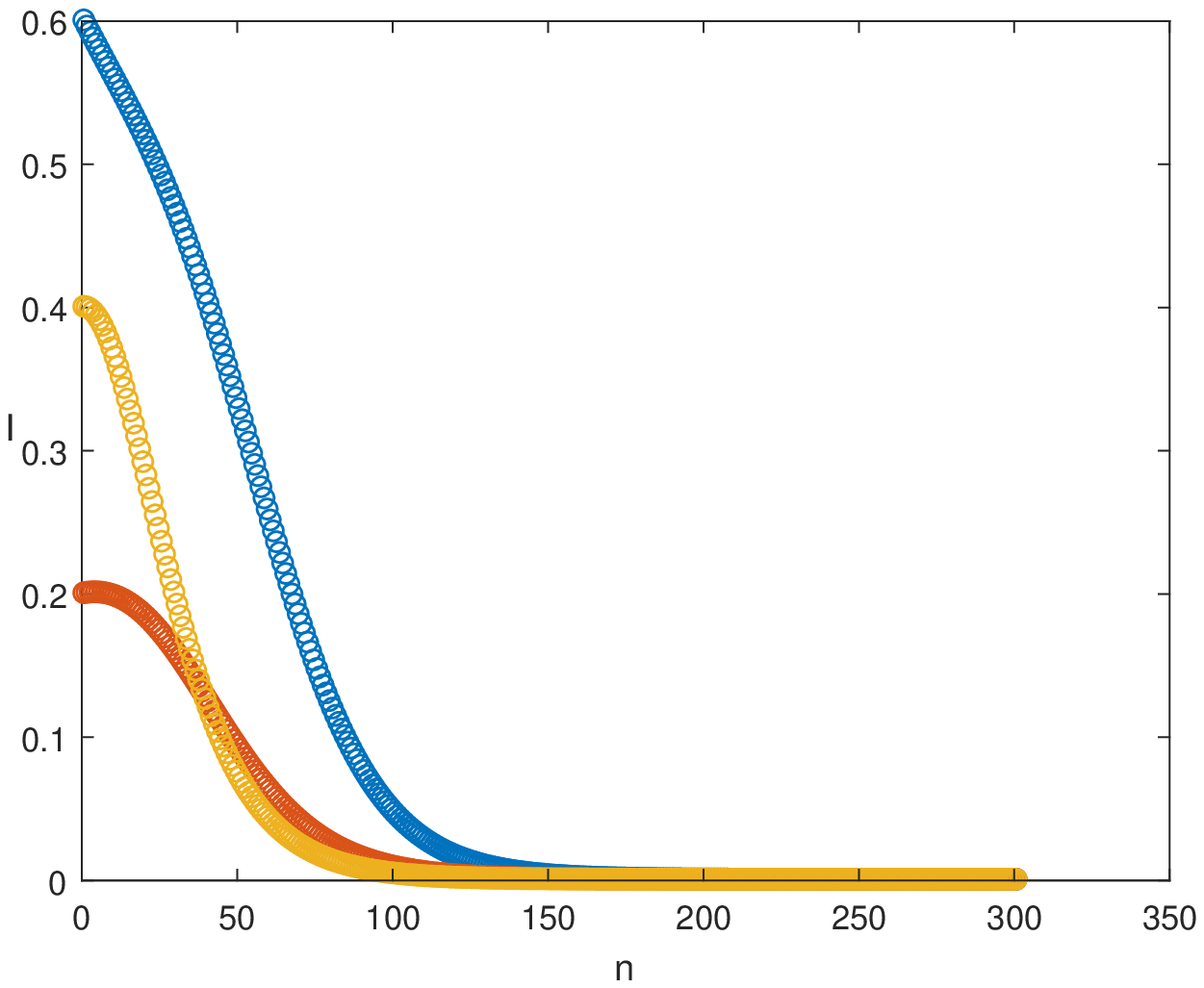}
  \end{minipage}
  \begin{minipage}[b]{.32\linewidth}
        \includegraphics[width=\linewidth]{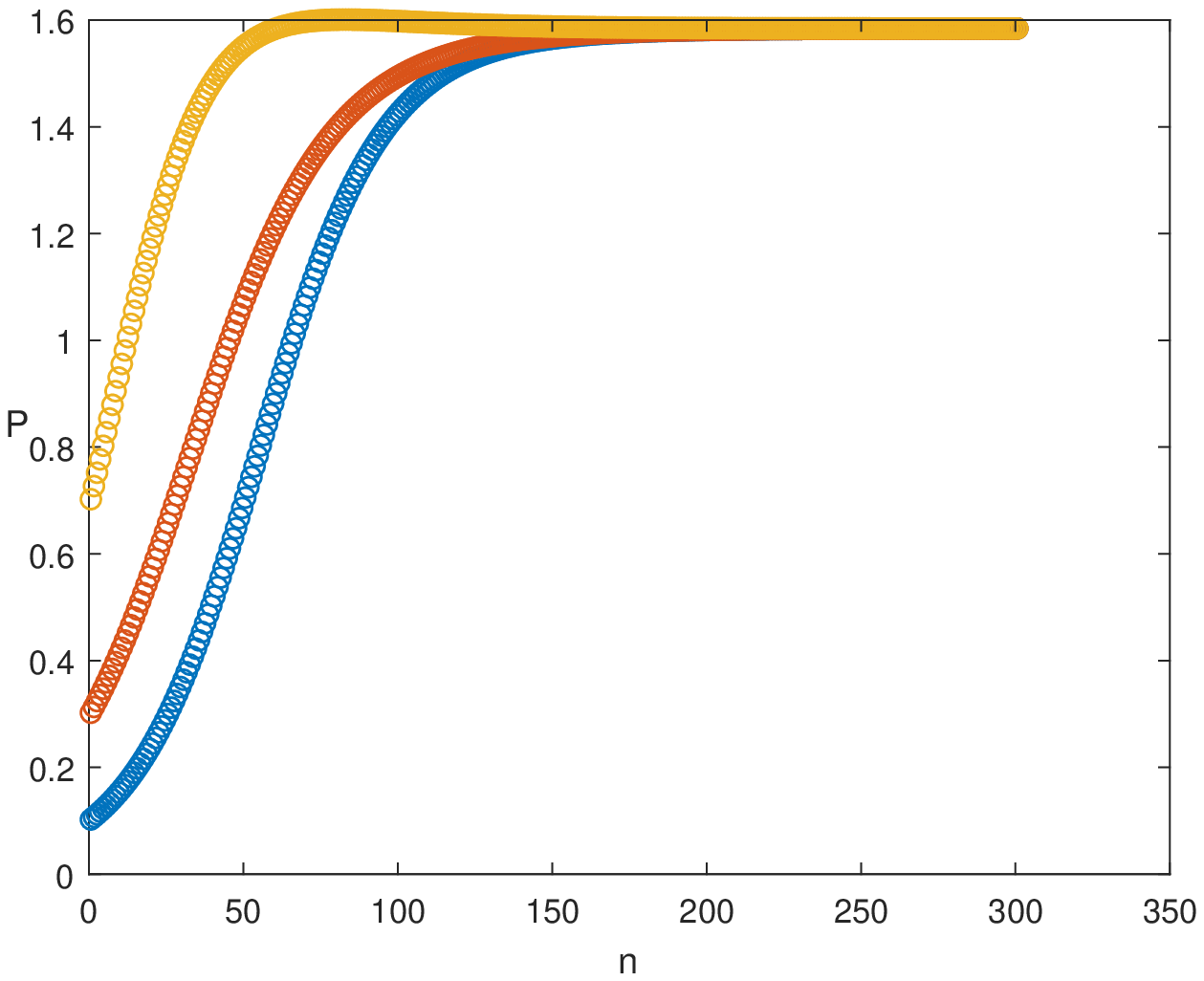}
  \end{minipage}
    \caption{Extinction, when $\beta=0.17$.}
      \label{fig-no-autonomous-DF-1}
\end{figure}
\begin{figure}
  \begin{minipage}[b]{.32\linewidth}
    \includegraphics[width=\linewidth]{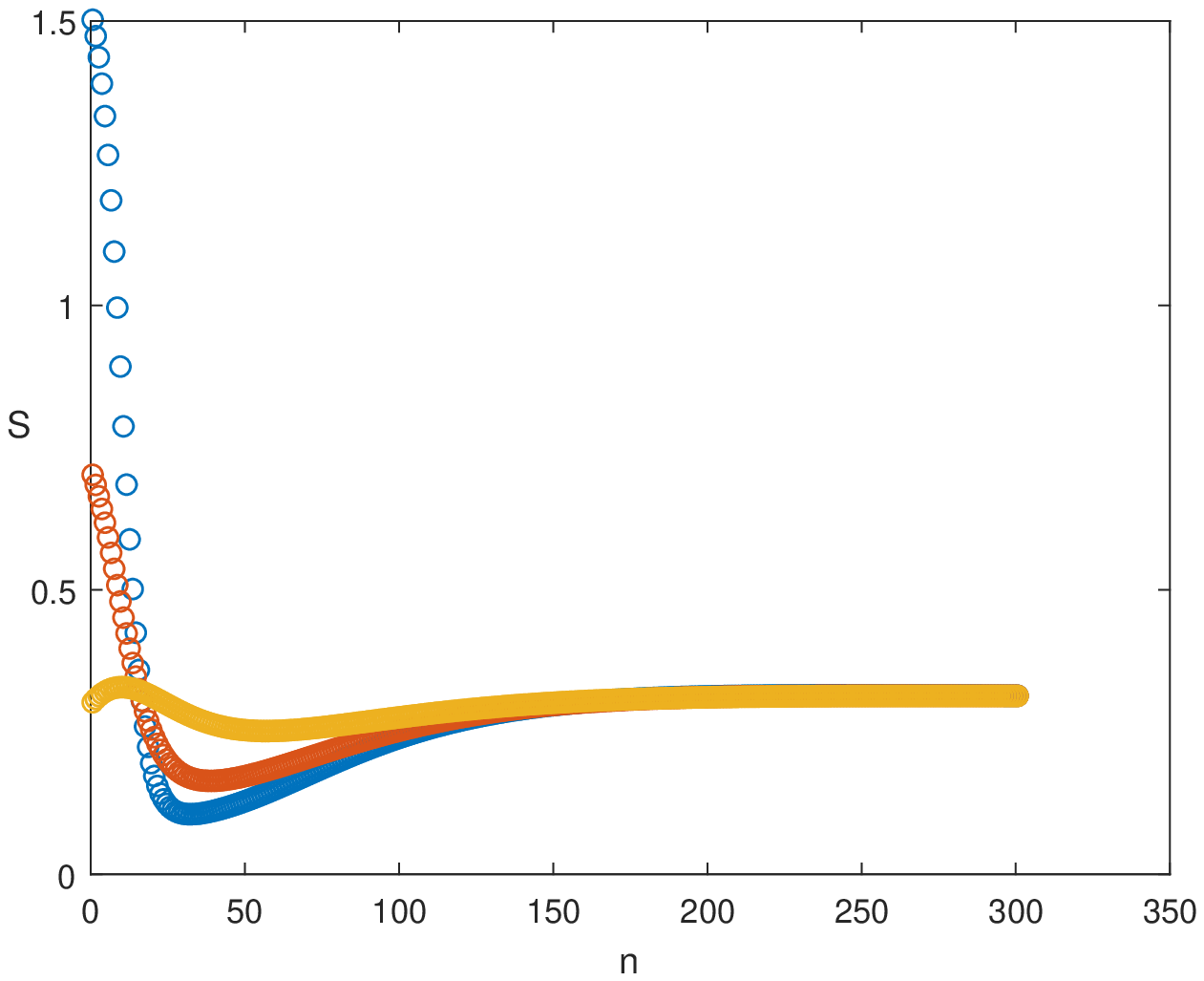}
  \end{minipage}
  \begin{minipage}[b]{.32\linewidth}
        \includegraphics[width=\linewidth]{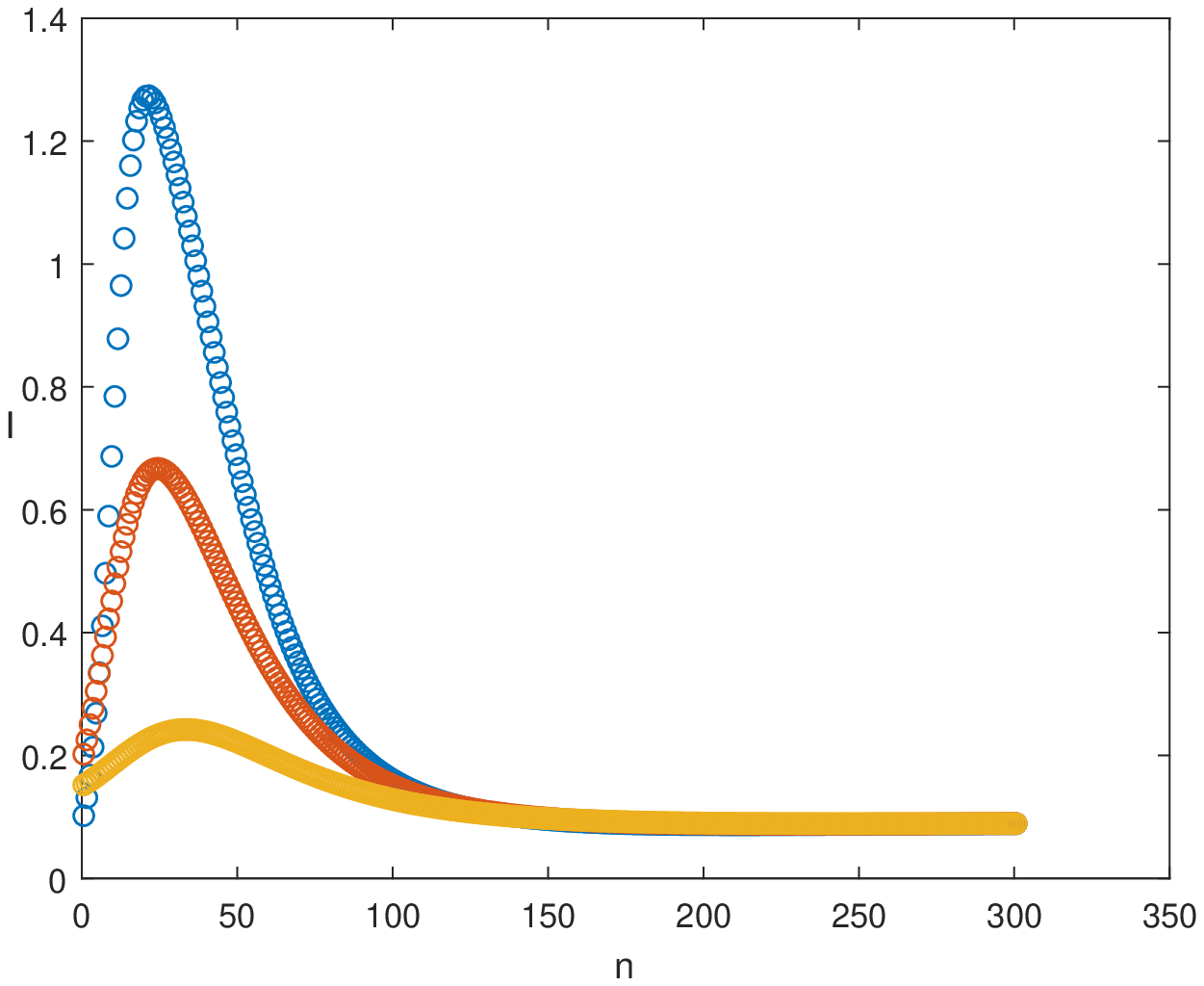}
  \end{minipage}
  \begin{minipage}[b]{.32\linewidth}
        \includegraphics[width=\linewidth]{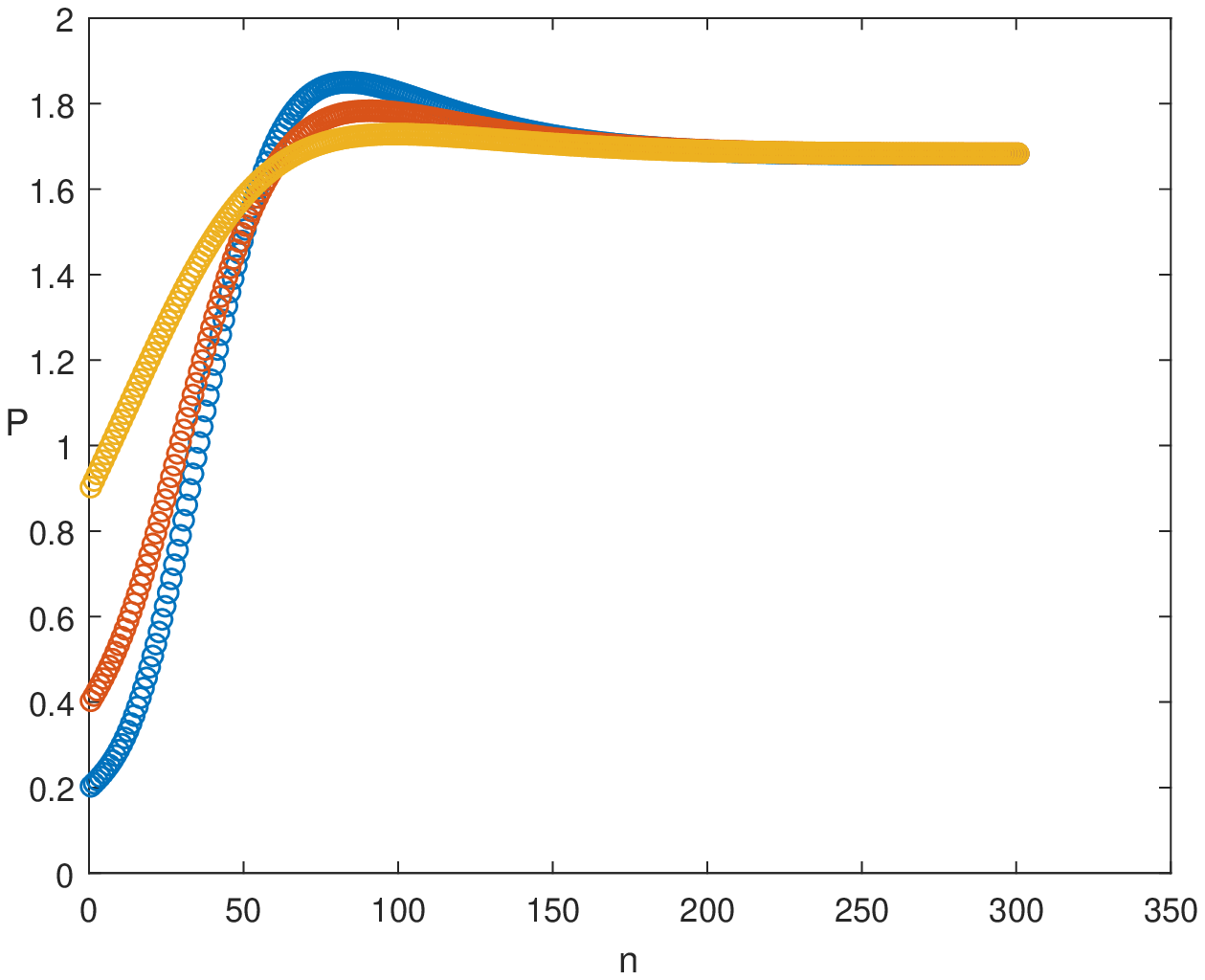}
  \end{minipage}
    \caption{Strong persistence, when $\beta=2.2$.}
      \label{fig-no-autonomous-PER-1}
\end{figure}

\bibliographystyle{elsart-num-sort}

\end{document}